\documentclass[reqno, 11pt]{amsart} 
\usepackage{amsfonts, amsmath, amssymb, amsthm}
\usepackage[margin=2.75cm, heightrounded]{geometry}
\usepackage{fancybox}
\usepackage{hhline, float}
\usepackage{amsmath}
\usepackage{mathtools}
\usepackage{mathrsfs}
\usepackage{nccmath}
\usepackage[dvipsnames]{xcolor}
\usepackage[colorlinks=true]{hyperref}
\hypersetup{citecolor=red, linkcolor=blue}

\newtheorem{theorem}{Theorem}[section]
\newtheorem{lemma}[theorem]{Lemma}

\theoremstyle{definition}
\newtheorem{definition}[theorem]{Definition}

\newtheorem{corollary}[theorem]{Corollary}
\newtheorem{proposition}[theorem]{Proposition}

\newtheorem*{proof of Tin}{Proof of Theorem \ref{Tin}}
\newtheorem*{proof of 2 of A-D}{Proof of Part (2) of Theorem \ref{A-D}}
\newtheorem*{proof of extremal of Tin}{Proof of Theorem \ref{extremal of Tin}}
\newtheorem*{proof of extremal of Tin A D}{Proof of Theorem \ref{extremal of A-D}}

\newtheorem{remark}[theorem]{Remark}

\numberwithin{equation}{section}

\newcommand{\NN}{\ensuremath{\mathbb{N}}}

\newcommand{\RR}{\ensuremath{\mathbb{R}}}
\newcommand{\CC}{\ensuremath{\mathbb{C}}}

\newcommand{\md}{\mathrm{\,d}}

\newcommand{\oo}{\infty}

\newcommand{\pt}{\partial}

\allowdisplaybreaks
\numberwithin{equation}{section}

\makeatletter
\@namedef{subjclassname@2020}{%
\textup{2020} Mathematics Subject Classification}
\makeatother

\begin{document}
\title[Sharp higher-order uncertainty principles]
{Sharp higher-order uncertainty principles}



\author{Hongtao Hu}
\address{Department of Mathematical Sciences, Tsinghua University. No.\,1 Tsinghua Garden, Beijing 100084, People's Republic of China.}
\email{huht23@mails.tsinghua.edu.cn}
\thanks{This manuscript was submitted for publication on September 10, 2026.} 

\author{Meiqi Liu}
\address{School of Mathematical Sciences, Hangzhou Dianzi University, Hangzhou 310018, China.}
\email{liumq@hdu.edu.cn}

\author{Wenming Zou}
\address{Department of Mathematical Sciences, Tsinghua University. No.\,1 Tsinghua Garden, Beijing 100084, People's Republic of China.}
\email{zou-wm@mail.tsinghua.edu.cn}

\begin{abstract}
{\rm We establish a class of sharp higher-order inequalities related to the uncertainty principle. First, by means of the "expanding the squares" method, we give a concise proof of sharp higher-order Heisenberg uncertainty principles. We provide explicit expressions for the optimal constants and establish the existence of extremal functions. Furthermore, we obtain sharp higher-order Heisenberg uncertainty principles for curl-free vector fields. In particular, when $N=2$, we give an answer to the higher-order version of Open Problem 9 raised by Maz’ya in [Integr. Equ. Oper. Theory (2018) 90:25]. Second, we establish sharp higher-order Heisenberg uncertainty principles that do not involve higher-order tensors. We also prove sharp higher-order Heisenberg uncertainty principles and higher-order hydrogen uncertainty principles for radially symmetric functions.
}
\end{abstract}


\maketitle


\vspace{4\baselineskip}

\section{Introduction}\label{sec 1}

	\noindent According to the Heisenberg-Pauli-Weyl uncertainty principle in quantum mechanics, it is impossible to accurately measure both the position $x$ and the momentum $p$ of a particle simultaneously. Heisenberg’s groundbreaking paper \cite{Heisenberg1927} introduced these physical ideas. The rigorous mathematical formulation was subsequently given by Kennard \cite{Kennard1927} and by Weyl and Pauli \cite{Weyl1931}. For a quantum state described by a normalized wave function $\psi(x)$ satisfying  $\int_{-\oo}^{+\oo}|\psi(x)|^2\md x=1$,  the mean values of $x$ and $p$ can be shifted to zero without loss of generality. The variances of $x$ and $p$ then take the form
	\begin{align*}
	    (\Delta x)^2=\int_{-\oo}^{+\oo}x^2|\psi(x)|^2\md x,\qquad (\Delta p)^2=h^2\int_{-\oo}^{+\oo}|\psi'(x)|^2\md x.
	\end{align*}
	where $h$ is the Planck's constant of action. In \cite{Weyl1931}, Weyl and Pauli proved that
	\begin{align*}
	    \Delta p\cdot\Delta x\geq \frac{h}{2}.
	\end{align*}
	This one-dimensional inequality can be generalized to $\RR^N$ $(N\in \NN_+)$ as follows
	\begin{align*}
	    \int_{\RR^N}|\nabla u|^2\md x\int_{\RR^N}|x|^2|u|^2\md x\geq \dfrac{N^2}{4}\left(\int_{\RR^N}|u|^2\md x\right)^2.
	\end{align*}
	This inequality holds for any function $u\in C^\infty_0\left(\RR^N\right)$.
	The same inequality also remains true for functions in the Schwartz space $\mathcal{S}\left(\RR^N\right)$ and in a suitable Sobolev space. Moreover, the constant $\dfrac{N^2}{4}$ is optimal and it is attained by the Gaussian profiles of the form $u(x)=a\mathrm{e}^{-b|x|^2}$ for $a\in\RR\backslash\{0\}$ and $b>0$, which belongs to the Schwartz space $\mathcal{S}\left(\RR^N\right)$ but not $C^\infty_0\left(\RR^N\right)$. The Heisenberg uncertainty principle is an important inequality in harmonic analysis and 
its time‑frequency counterparts are critical for signal and image processing.
	
	In fact, the Heisenberg uncertainty principle can be viewed as a special case of the following interpolation inequality established by Caffarelli, Kohn and Nirenberg \cite{MR768824}, which we shall refer to as the CKN inequality for brevity: 
	\begin{equation*}
	    \int_{\RR^N}|\nabla u|^2\md x\int_{\RR^N}\frac{|u|^{2p-2}}{|x|^{2q-2}}\md x\geq\frac{(N-q)^2}{p^2}\left(\int_{\RR^N}\frac{|u|^p}{|x|^q}\md x\right)^2 \qquad\text{for $u\in C^\oo_0\left(\RR^N\right)$},
	\end{equation*}
	where $0<q<2<p$ and  $2<N<\dfrac{2(p-q)}{p-2}$. The constant $\dfrac{(N-q)^2}{p^2}$ is attained by 
	\begin{equation*}
	    u(x)=\left(b+|x|^{2-q}\right)^{\frac{1}{2-p}}\qquad\text{for }b>0.
	\end{equation*}
	The CKN inequality reduces to the  Heisenberg uncertainty principle when $p\to2$ and $q\to0$, to the Hardy inequality when $p\to2$ and $q\to2$, and to the following hydrogen uncertainty principle when $p\to2$ and  $q\to1$:
	\begin{align*}
	    \int_{\RR^N}\left|\nabla u\right|^2\md x\int_{\RR^N}\left|u\right|^2\md x\geq \frac{(N-1)^2}{4}\left(\int_{\RR^N}\frac{|u|^2}{|x|}\md x\right)^2,
	\end{align*}
	\\
	Here the constant $\dfrac{(N-1)^2}{4}$ is attained by $u(x)=a\mathrm{e}^{-b|x|}$ for $a\in\RR\backslash\{0\}$ and $b>0$. The hydrogen uncertainty principle provides a lower bound for the energy functional of a single hydrogenic atom (see \cite[Section 2.2]{MR2583992}).   It can be used to study the stability of a hydrogenic atom in a magnetic field as discussed in \cite{MR836003,MR2583992}.
	
	By Cauchy inequality, one immediately obtains the vector field version of the  Heisenberg uncertainty principle. In \cite[Section  3.9]{MR3795177}, Maz'ya raised an open problem: determine the best constant $\mu(N)$ in the following inequality
	\begin{align*}
	    \int_{\RR^N}|\nabla \vec{U}|^2\md x\int_{\RR^N}|x|^2|\vec{U}|^2\md x\geq \mu(N)\left(\int_{\RR^N}|\vec{U}|^2\md x\right)^2,
	\end{align*}
	\\
where $\vec{U}$ is an arbitrary divergence-free vector field. C. Cazacu, J. Flynn and N. Lam \cite{MR4470731} solved this open problem for the case $N=2$, while Hamamoto \cite{MR4556904} solved it for $N\geq 3$. In particular, \cite{MR4470731} considered the case where the vector field $\vec{U}=\nabla u$ for $u\in C^\infty_0\left(\RR^N\right)$ and  established the second-order Heisenberg  uncertainty principle:
\begin{align*}
    \int_{\RR^N}|\Delta u|^2\md x\int_{\RR^N}|x|^2|\nabla u|^2\md x\geq\frac{(N+2)^2}{4}\left(\int_{\RR^N}|\nabla u|^2\md x\right)^2,
\end{align*}
for $u\in \mathcal{S}\left(\RR^N\right)$ $(N\geq1)$ and the constant $\dfrac{(N+2)^2}{4}$ is attained by $u(x)=a\mathrm{e}^{-b|x|^2}$ for $b>0$. Here and henceforth, we denote by $\Delta$ the Laplacian operator in the Euclidean space $\RR^N$. Subsequently, \mbox{C. Cazacu}, J. Flynn and N. Lam \cite{MR4565040} established a class of sharp weighted inequalities of CKN type for arbitrary curl‑free vector fields, as well as sharp second‑order CKN inequalities for scalar fields. Duong and Nguyen \cite{MR4902771} proved some sharp second-order CKN inequalities by developing a new technique different from that used in \cite{MR4470731}. In addition, the stability results concerning the  Heisenberg uncertainty principle were obtained in \cite{ MR4255507,MR4151199}. Later, the authors of \cite{MR4682470} obtained sharp stability for the CKN inequlity and uncertainty principle. The stability estimate for second-order Heisenberg uncertainty principle was given in \cite{MR4886055}, and subsequently improved in \cite{MR5008481}, where sharp stability estimates were obtained for the second‑order Heisenberg uncertainty principle with scalar fields, as well as sharp stability of the first‑order  Heisenberg uncertainty  principle with  curl‑free vector fields.   The stability of the Heisenberg uncertainty principle with monomial weights was established by \cite{MR5013418}. In \cite{lam2026heisenberguncertaintyprinciplehalf}, the authors proved the sharp
 Heisenberg uncertainty principle on the half spaces or orthants and also established the corresponding stability estimates.

There are also some results concerning the  Heisenberg uncertainty principle on manifolds. Krist\'aly \cite{MR3862150} examined the influence of curvature on the uncertainty principle and studied the sharp  Heisenberg uncertainty principle on both Cartan–Hadamard manifolds and manifolds with non‑negative Ricci curvature. The sharp Heisenberg uncertainty principle on Finsler manifolds was established in \cite{MR4169684}. In \cite{MR4417395}, the authors gave the first example of second order Heisenberg uncertainty principle in the hyperbolic space. For CKN inequalities on manifolds with negative curvature, we refer to \cite{MR4591784,MR4803460,MR4383243,MR4341617}.

As for the hydrogen uncertainty principle, a second‑order version was given in \cite{MR4470731}  as follows:
\begin{align*}
    \int_{\RR^N}\left|\Delta u\right|^2\md x\int_{\RR^N}\left|\nabla u\right|^2\md x\geq \frac{(N+1)^2}{4}\left(\int_{\RR^N}\frac{|\nabla u|^2}{|x|}\md x\right)^2\quad\text{for $N\geq 5$ and $u\in W^{2,2}\left(\RR^N\right)$},
\end{align*}
where the extremal functions are given by $u(x)=(1+b|x|)\mathrm{e}^{-b|x|}$ for $b>0$. In \cite{MR4470731}, it was conjectured that this result also holds for $2\leq N\leq 4$. Recently, the conjecture was shown to be false for $N=2,\,3$ by Chen and Tang \cite{chen2025extremalfunctionssecondorder}, while Huang and Ye \cite{huang2025sharpheisenberguncertaintyprinciple} proved its validity for $N=4$ and gave a unified proof for $N\geq4$. In addition, it was proved in \cite{MR4470731} that if the derivative is replaced by the radial derivative (see \eqref{def of radial derivative}), then the hydrogen uncertainty principle holds for all $N\geq2$.

In this paper, we primarily aim to establish higher-order uncertainty principles and to investigate their best constants. Furthermore, we consider the vector field setting and provide an answer to the higher-order version of Maz’ya’s Open Problem 9 in \cite{MR3795177} for the case $N=2$. Inspired by \cite{MR4470731}, we observe that the best constant increases from $\dfrac{N^2}{4}$
 to $\dfrac{(N+2)^2}{4}$ when passing from the first-order Heisenberg uncertainty principle to the second-order Heisenberg  uncertainty principle. A natural question then arises: for higher-order Heisenberg  uncertainty principles, will the corresponding optimal constants continue to increase? Before presenting the main conclusions, we utilize the Cauchy inequality to deduce the form of the higher-order Heisenberg uncertainty principle. From now on, we define $\Delta^0u=u$ and  $\Delta^j=\Delta\left(\Delta^{j-1}\right)$ for nonnegative integer $j$. We set
\begin{align}
\label{def of Pm}
    P_m u=\begin{cases}\Delta^j u\,,\qquad&\text{ if }m=2j,\\
    \nabla\Delta^j u\,,&\text{ if }m=2j+1.
    \end{cases}
\end{align}
 For $k\in\NN_+$, let $\nabla^k u$ denote the $k$-th rank tensor consisting of all $k$-th order partial derivatives of $u$ and $\nabla^0u=u$. Let $|\nabla^k u|$ denote the Frobenius norm defined as the square root of the sum of the squares of the elements of $\nabla^k u$. Suppose that for $u\in C^\infty_0\left(\RR^N\right)$, the following m-th order Heisenberg uncertainty principle holds:
\begin{align*}
    \int_{\RR^N}|P_mu|^2\md x\int_{\RR^N}|x|^2|\nabla^{m-1}u|^2\md x\geq \mu_m(N)\left(\int_{\RR^N}|P_{m-1}u|^2\md x\right)^2.
\end{align*}
For $\vec{U}=\left(U_1,\,\dots,\,U_N\right)\in C^\oo_0\left(\RR^N;\,\RR^N\right)$, define
\begin{equation*}
    P_m\vec{U}=\left(P_m U_1,\,\dots,\,P_m U_N\right),\qquad \nabla^{m-1}\vec{U}=\left(\nabla^{m-1}{U_1},\,\dots,\,\nabla^{m-1}{U_N}\right).
\end{equation*}
Then we have
\begin{align*}
    \int_{\RR^N}\left|P_m\vec{U}\right|^2\md x\int_{\RR^N}|x|^2\left|\nabla^{m-1}\vec{U}\right|^2\md x&=\left(\sum_{i=1}^N\int_{\RR^N}\left|P_m U_i\right|^2\md x\right)\left(\sum_{i=1}^N\int_{\RR^N}|x|^2\left|\nabla^{m-1}{U_i}\right|^2\md x\right)\\
    &\geq\left(\sum_{i=1}^N\sqrt{\int_{\RR^N}\left|P_m U_i\right|^2\md x\int_{\RR^N}|x|^2\left|\nabla^{m-1}{U_i}\right|^2\md x}\right)^2\\
    &\geq \mu_m(N)\left(\sum_{i=1}^N\int_{\RR^N}\left|P_{m-1}U_i\right|^2 \md x\right)^2\\
    &=\mu_m(N)\left(\int_{\RR^N}\left|P_{m-1}\vec{U}\right|^2 \md x\right)^2.
\end{align*}
If we replace $\vec{U}$ by $\nabla u$ for $u\in C^\oo_0\left(\RR^N\right)$, then we obtain the (m+1)-th order Heisenberg  uncertainty principle:
\begin{align*}
     \int_{\RR^N}|P_{m+1}u|^2\md x\int_{\RR^N}|x|^2|\nabla^{m}u|^2\md x\geq \mu_m(N)\left(\int_{\RR^N}|P_{m}u|^2\md x\right)^2.
\end{align*}
From now on, we denote by $\mathcal{H}^m$ the completion of  $C^\oo_0\left(\RR^N\right)$ with respect to the seminorm
\begin{align}
\label{def of seminorm}
   \lVert u\rVert_{\mathcal{H}^m}^2= \int_{\RR^N}|P_mu|^2\md x+\int_{\RR^N}|x|^2|\nabla^{m-1}u|^2\md x.
\end{align}
The first result of this paper is the following inequality, which can be regarded as a higher-order version of Theorem 2.1 in \cite{MR4470731}.
\begin{theorem}
\label{thm m order HUP}
	For $N,\,m\in\NN_+$, $m\geq2$ and $u\in\mathcal{H}^m$, we have the following inequality
	\begin{equation}
	\label{eq m order HUP}
	    \int_{\RR^N}|P_mu|^2\md x\int_{\RR^N}|x|^2|\nabla^{m-1}u|^2\md x\geq \mu_m(N)\left(\int_{\RR^N}|P_{m-1}u|^2\md x\right)^2,
	\end{equation}
	where $P_m$ is defined by \eqref{def of Pm} and the constant $\mu_m(N)$ is given by \begin{align}
	\label{def of mu m N}
	    \mu_m(N)=\left(\frac{2+\sqrt{N^2+4(m-2)(N-1)}}{2}\right)^2.
	\end{align}
\end{theorem}
Furthermore, we prove that the constant $\mu_m(N)$ appearing in \eqref{eq m order HUP} is optimal and can be attained. In order to obtain the form of the extremal function, we need to use the confluent hypergeometric function defined as follows (see \cite{MR3586206}):
\begin{align*}
    _{1}F_{1}(\alpha,\,\gamma,\,z)\coloneqq\sum_{n=0}^\oo \frac{(\alpha)_n}{(\gamma)_n}\frac{z^n}{n!}\qquad\left(z\in\CC;\,\gamma\not=0,\,-1,\,-2,\,\dots\right),
\end{align*}
where $(\alpha)_n$ and $(\gamma)_n$ are Pochhammer symbols. It is known that ${_{1}F_{1}(\alpha,\,\gamma,\,z)}$ is an entire function with respect to $z$. For a detailed definition and the properties of the hypergeometric function $_1F_1$, we refer to Section \ref{sec 2}.
\vspace{1cm}
\begin{theorem}
\label{extremal}
For $N\geq1$ and $m\geq2$, the constant $\mu_m(N)$ in \eqref{eq m order HUP} is optimal. Moreover, if $N\geq2$, then  $\mu_m(N)$ is attained by some radially symmetric function $\Phi_{b,m}\in \mathcal{H}^m$ satisfying
\begin{align*}
    P_{m-2}\Phi_{b,m}(x)=\begin{cases}
    \mathrm{e}^{-\frac{br^2}{2}}{_1F_1}\left(\dfrac{N-\sqrt{N^2+4(m-2)(N-1)}}{4},\,\dfrac{N}{2},\,\dfrac{br^2}{2}\right),\,&\text{if $m$ is even,}\\
     x\mathrm{e}^{-\frac{br^2}{2}}{_1F_1}\left(\dfrac{2+N-\sqrt{N^2+4(m-2)(N-1)}}{4},\,\dfrac{N+2}{2},\,\dfrac{br^2}{2}\right),\,&\text{if $m$ is odd,}
    \end{cases}
\end{align*}
\text{where} $r=|x|$ and $b>0$. If $N=1$, then $\mu_m(N    )$ is attained by some function $\Phi_{b,m}\in\mathcal{H}^m$ such that $P_{m-2}\Phi_{b,m}=\mathrm{e}^{-br^2}$ for $b>0$.
\end{theorem}
\begin{remark}
  In the special case $m=2$, the hypergeometric function $_1F_1$ reduces to a constant. Consequently, the extremal functions $\Phi_{b,m}$ are precisely the Gaussian profile. This coincides with the result related to the second-order Heisenberg uncertainty principle in  \cite{MR4470731}.
\end{remark}

It is worth mentioning that, in order to prove Theorem \ref{thm m order HUP} and Theorem \ref{extremal}, we do not use the spherical harmonics decomposition method employed in \cite{MR4470731}, since the computations would become cumbersome as the order increases. Instead, we adopt the so‑called “expanding the squares” method which greatly simplifies the computations. This idea has also been employed in \cite{MR4315551,MR4565040,MR2356072,MR4902771} to handle integrals involving the first and second order derivatives. We shall present a method for treating the higher-order case. More precisely, for a given order $m$, we define the operators $\mathcal{E}_{b,c}$ and $\mathcal{O}_{b,c}$ as in \eqref{def of Ebc} and \eqref{def of Obc}, respectively, where $\mathcal{E}_{b,c}$ corresponds to even $m$ and $\mathcal{O}_{b,c}$  corresponds to odd $m$. Then, by  expanding the corresponding $L^2$ norms and using the crucial integral identities given in Section \ref{sec 2}, we overcome the difficulties caused by integrals of higher-order derivatives and prove Theorem \ref{thm m order HUP}. As for the existence of extremal functions, the method of expanding the squares leads us to begin from $\mathcal{E}_{b,c}(v)=0$ and $\mathcal{O}_{b,c}(v)=0$. By constructing the corresponding solutions and analyzing their asymptotic behavior carefully, we establish the existence of extremal functions.

Analogous to the treatment in \cite{MR3528527,MR4470731,MR4565040,MR5008481}, we can regard the smooth vector field $\vec{U}\in C^\oo_0\left(\RR^N;\,\RR^N\right)$ as a 1-form $U=\sum_{i=1}^N U_i \md x_i$. In addition, when $N\geq2$, we may use the exterior derivative to define a 2-form
\begin{align*}
    \mathrm{curl }\,\vec{U}=\md U=\sum_{i=1}^N\sum_{j=1}^N \left(\frac{\partial U_i}{\partial x^j}\md x_j\right)\wedge \md x_i. 
\end{align*}
If $\md U=0$, then $U$ is called a  closed 1-form. By Poincar\'{e} lemma, $U$ is a closed 1‑form if and only if it is exact, i.e., there exists a function $u\in C^\oo_0\left(\RR^N\right)$ such that $\vec{U}=\nabla u$. Consequently, by Theorem \ref{thm m order HUP} and Theorem \ref{extremal}, we have the following result concerning higher‑order uncertainty principles for curl‑free vector fields.
\begin{corollary}
\label{cor m order HUP}
Let $N,\,m\in\NN_+$ with $m\geq2$, and let $\vec{U}\in C^\oo_0\left(\RR^N;\,\RR^N\right)$. If $N\geq2$, we additionally require that  $\mathrm{curl }\,\vec{U}=0$, or equivalently, $U=\sum_{i=1}^N U_i \md x_i$ is a closed $1$-form. Then we have the following inequality
	\begin{equation}
	\label{equation of cor m HUP}
	    \int_{\RR^N}|P_{m-1}\vec{U}|^2\md x\int_{\RR^N}|x|^2|\nabla^{m-2}\vec{U}|^2\md x\geq \mu_{m}(N)\left(\int_{\RR^N}|P_{m-2}\vec{U}|^2\md x\right)^2,
	\end{equation}
	where $\mu_m(N)$ is defined by \eqref{def of mu m N} and is optimal in \eqref{equation of cor m HUP}.
\end{corollary}

When $N=2$, the divergence-free vector field $\vec{U}\in C^\oo_{0}\left(\RR^2;\,\RR^2\right)$ can be expressed in terms of a function $u\in C^\oo_0\left(\RR^2\right)$ as  $\vec{U}=\left(\dfrac{\pt u}{\pt x_2},\,-\dfrac{\pt u}{\pt x_1}\right)$. In this case, we can give an answer to the higher-order version of Open Problem 9 raised by Maz'ya in \cite{MR3795177}.
\begin{corollary}
\label{cor Mazya}
Let $N=2$, $m\in\NN_+$ and $m\geq2$. For any $\vec{U}\in C^\oo_0\left(\RR^2;\,\RR^2\right)$ satisfying $\mathrm{div}\,\vec{U}=0$, we have
	\begin{equation}
	\label{equation of cor Mazya}
	    \int_{\RR^N}|P_{m-1}\vec{U}|^2\md x\int_{\RR^N}|x|^2|\nabla^{m-2}\vec{U}|^2\md x\geq \mu_m(2)\left(\int_{\RR^N}|P_{m-2}\vec{U}|^2\md x\right)^2.
	\end{equation}
	Moreover, the constant $\mu_m(2)=\left(1+\sqrt{m-1}\right)^2$ in \eqref{equation of cor Mazya} is optimal.
\end{corollary}

Theorem \ref{thm m order HUP} presents a higher‑order Heisenberg uncertainty principle involving the $(m-1)$-th order tensor $\nabla^{m-1}u$. Next, we give a version of the Heisenberg uncertainty principle that does not involve higher‑order tensors. In this situation, the optimal constant does not keep improving as the order increases. We write  $\mathcal{H}^m_1$ the completion of  $C^\oo_{0}\left(\RR^N\right)$ with respect to the seminorm
\begin{align}
\label{def of seminorm 1}
   \lVert u\rVert_{\mathcal{H}^m_1}^2= \int_{\RR^N}|P_m u|^2\md x+\int_{\RR^N}|x|^2|P_{m-1}u|^2\md x.
\end{align}
\begin{theorem}
\label{thm m order HUP without tensor}
Let $N,\,m\in \NN_+$ such that $N\geq2$ and $m\geq2$. Then we have
\begin{align}
     \int_{\RR^N}|P_mu|^2\md x\int_{\RR^N}|x|^2|P_{m-1}u|^2\md x\geq \mu^\star_m(N)\left(\int_{\RR^N}|P_{m-1}u|^2\md x\right)^2\quad\text{for $u\in\mathcal{H}^m_1$},
     \label{equation m order HUP without tensor}
\end{align}
where $\mu^\star_m(N)=\dfrac{(N+2)^2}{4}$ if $m$ is even and $\mu_m^\star(N)=\dfrac{N^2}{4}$ if $m$ is odd. Moreover, the constant $\mu_m^\star(N)$ is attained by some radially symmetric function $\Phi_{b,m}^\star\in\mathcal{H}^m_1$.
\end{theorem}

Note that when $N=1$, we have $P_{m-1}u=\nabla^{m-1}u$. Hence, by Theorem  \ref{thm m order HUP} and Theorem \ref{extremal}, the optimal constant in Theorem \ref{thm m order HUP without tensor} for $N=1$ is $\mu_m^\star(1)=\dfrac{9}{4}$. As a corollary of Theorem \ref{thm m order HUP without tensor}, we obtain a higher‑order Heisenberg uncertainty principle with radial derivatives. More precisely, define
\begin{equation}
    \label{def of radial derivative}
    \mathcal{R}_1\coloneqq \frac{x}{|x|}\cdot\nabla,\qquad\mathcal{R}_2\coloneqq \mathcal{R}_1\circ\mathcal{R}_1+\frac{N-1}{r}\mathcal{R}_1.
\end{equation}
Let $C^\oo_{0,\text{rad}}\left(\RR^N\right)=\left\{u\in C^\oo_{0}\left(\RR^N\right):\,u\text{ is radially symmetric}\right\}$. We denote by $\mathcal{H}^m_{2}$ the completion of  $C^\oo_{0,\text{rad}}\left(\RR^N\right)$ with respect to the seminorm
\begin{align}
\label{def of seminorm radial}
   \lVert u\rVert_{\mathcal{H}^m_{2}}^2= \int_{\RR^N}|\mathcal{R}_m u|^2\md x+\int_{\RR^N}|x|^2|\mathcal{R}_{m-1}u|^2\md x,
\end{align}
where
\begin{align}
\label{def of Rm}
    \mathcal{R}_m u=\begin{cases}\mathcal{R}_2^j u\,,\qquad&\text{ if }m=2j,\\
    \mathcal{R}_1\mathcal{R}_2^j u\,,&\text{ if }m=2j+1.
    \end{cases}
\end{align}
Here $\mathcal{R}_2^j\coloneqq\mathcal{R}_2\circ\mathcal{R}_2^{j-1}$ and $\mathcal{R}^0_2u=u.$
\begin{corollary}
\label{cor m order HUP without tensor}
Let $N,\,m\in \NN_+$ such that $N\geq2$ and $m\geq2$. Then we have
\begin{align}
     \int_{\RR^N}|\mathcal{R}_mu|^2\md x\int_{\RR^N}|x|^2|\mathcal{R}_{m-1}u|^2\md x\geq \mu^\star_m(N)\left(\int_{\RR^N}|\mathcal{R}_{m-1}u|^2\md x\right)^2\quad\text{for $u\in\mathcal{H}^m_2$},
     \label{equation cor m order HUP without tensor}
\end{align}
where $\mathcal{R}_m$ is defined by \eqref{def of Rm} and $\mu^\star_m(N)$ is given in Theorem \ref{thm m order HUP without tensor}. Moreover, the constant $\mu_m^\star(N)$ is attained by some radially symmetric function $\Phi_{b,m}^\star\in\mathcal{H}^m_2$.
\end{corollary}
\begin{remark}
If $u$ is not radially symmetric, then the integral $\int_{\RR^N}\left|\mathcal{R}_mu\right|^2\md x$ may diverge. For example, when $N=m=2$, taking $u(x)=x_1\mathrm{e}^{-|x|^2}$, we have $\int_{\RR^2}\left|\mathcal{R}_2 u\right|^2\md x=+\oo$.
\end{remark}

Finally, we give the higher-order hydrogen uncertainty principles for radially symmetric functions. Let $\mathcal{H}^m_3$ be the closure of $C^\oo_{0,\text{rad}}\left(\RR^N\right)$ with respect to the seminorm
\begin{align}
\label{def of seminorm radial HyUP}
   \lVert u\rVert_{\mathcal{H}^m_{3}}^2= \int_{\RR^N}|\mathcal{R}_m u|^2\md x+\int_{\RR^N}|\mathcal{R}_{m-1}u|^2\md x.
\end{align}
\begin{theorem}
\label{thm m order HyUP}
Let $N,\,m\in\NN_+$ and $N\geq 2$. Then we have
\begin{align}
\label{equation m order HyUP}
    \int_{\RR^N}\left|\mathcal{R}_m u\right|^2\md x\int_{\RR^N}\left|\mathcal{R}_{m-1}u\right|^2\md x\geq\bar{\mu}_m(N)\left(\int_{\RR^N}\frac{\left|\mathcal{R}_{m-1}u\right|^2}{|x|}\md x\right)^2 \qquad\text{for $u\in\mathcal{H}_3^m$},
\end{align}
where $\bar{\mu}_m(N)=\dfrac{(N+1)^2}{4}$ if $m$ is even and $\bar{\mu}_m(N)=\dfrac{(N-1)^2}{4}$ if $m$ is odd. Moreover, the constant $\bar\mu_m(N)$ is attained by some radially symmetric function $\Psi_{b,m}\in\mathcal{H}^m_3$.
\end{theorem}
Our paper is organized as follows:  In Section \ref{sec 2}, we present the properties of the hypergeometric functions needed in the sequel, together with some crucial integral identities that will be used to handle integrals of higher‑order derivatives.
In Section \ref{sec 3}, we carry out a detailed expansion of the $L^2$ norms of $\mathcal{E}_{b,c}(v)$ and $\mathcal{O}_{b,c}(v)$, and give the proof of Theorem \ref{thm m order HUP}. In Section \ref{sec 4}, we study the solutions of the equations $\mathcal{E}_{b,c}(v)=0$ and $\mathcal{O}_{b,c}(v)=0$ for  specific choice of $c$. We provide the explicit expressions of solutions and analyze their asymptotic behaviors. In Section \ref{sec 5}, we prove the existence of extremal functions for Heisenberg uncertainty principle and present the proofs of Theorem \ref{extremal}, Corollary \ref{cor m order HUP} and Corollary \ref{cor Mazya}. 
In Section \ref{sec 6}, we will prove Theorem \ref{thm m order HUP without tensor} and Corollary \ref{cor m order HUP without tensor}.
In Section \ref{sec 7}, we will study the hydrogen uncertainty principle and prove Theorem \ref{thm m order HyUP}.

Throughout this paper, the letter $C$ will denote positive constants that may vary from line to
line.
\vspace{2cm}

\section{Preliminaries}
\label{sec 2}
In this section, we first present the definition and some properties of the hypergeometric function $_1F_1$
, mainly following \cite{MR174795,MR3586206}. Afterwards, we give some integral identities, which will  be crucial in the derivation of the higher-order uncertainty principle.

The hypergeometric function $_1F_1(\alpha,\, \gamma,\,z)$ is defined as 
\begin{align*}
    _{1}F_{1}(\alpha,\,\gamma,\,z)\coloneqq\sum_{n=0}^\oo \frac{(\alpha)_n}{(\gamma)_n}\frac{z^n}{n!}\qquad\left(z\in\CC;\,\gamma\not=0,\,-1,\,-2,\,\dots\right),
\end{align*}
where Pochhammer symbols $(\alpha)_n$ are defined by
\begin{equation*}
    (\alpha)_0=1,\qquad (\alpha)_n=\alpha(\alpha+1)\cdots(\alpha+n-1)\qquad(n=1,\,2,\,3,\,\dots),
\end{equation*}
and similarly to $(\gamma)_n$. The function $F(z)={_{1}F_{1}(\alpha,\,\gamma,\,z)}$ is an entire function of $z$ and satisfies the confluent hypergeometric equation
\begin{equation*}
    zF''(z)+(\gamma-z)F'(z)-\alpha F(z)=0.
\end{equation*}
\vspace{1cm}
\begin{lemma}
\label{property of F}
For $\gamma\not=0,\,-1,\,-2,\,\dots$,  the function $_1F_1(\alpha,\,\gamma,\,z)$ satisfies the following property:
\begin{enumerate}
\item Differentiation formula (see \textup{\cite[Chapter 8]{MR3586206}}): $$\dfrac{\md}{\md z}{_1F_1}(\alpha,\,\gamma,\,z)=\dfrac{\alpha}{\gamma}{_1F_1}(\alpha+1,\,\gamma+1,\,z).$$ 
\item Asymptotic property (see \textup{\cite[Chapter 9]{MR174795}}): If $\alpha$ is either a negative integer or 0, then $_1F_1$ is a polynomial in $z$ of degree not exceeding $|\alpha|$. If $\alpha$ is neither a negative integer nor 0, then $_1F_1$ has the asymptotic representation:
\begin{align*}
_1F_1(\alpha,\,\gamma,\,z)=\dfrac{\Gamma(\gamma)}{\Gamma(\alpha)}\mathrm{e}^{z}z^{\alpha-\gamma}\left[\sum_{k=0}^n\dfrac{(\gamma-\alpha)_k(1-\alpha)_k}{k!}z^{-k}+O\left(|z|^{-n-1}\right)\right],
\end{align*}
for large $|z|$ and $|\mathrm{arg}\,z|\leq \dfrac{\pi}{2}-\delta$, where $n$ is an arbitrary nonnegative integer and $\delta>0$ is arbitrary small.
\item Kummer’s transformation formula  (see \textup{\cite[Chapter 8]{MR3586206}}): \begin{equation*}
_1F_1(\alpha,\,\gamma,\,z)=\mathrm{e}^z {_1F_1}(\gamma-\alpha,\,\gamma,\,-z)
.
\end{equation*}
\end{enumerate}
\end{lemma}
Next, we present some integral identities involving higher-order derivatives.
\begin{lemma}
\label{integral identity 1}
For $N,\,m\in\NN_+$ and $u\in C^\oo_0\left(\RR^N\right)$, we have
\begin{align*}
    \int_{\RR^N}\left|\nabla^m u\right|^2\md x=\int_{\RR^N}\left|P_m u\right|^2\md x.
\end{align*}
\end{lemma}
\begin{proof}
The case $m=1$ is trivial. Hence, we only consider $m\geq2$ in what follows. In fact, when $m\geq 2$, there holds:
\begin{align*}
    \int_{\RR^N}\left|\nabla^m u\right|^2\md x&=\int_{\RR^N}\sum_{i_1,\,i_2=1}^N\left(\sum_{i_3,\,\dots,\,i_m=1}^N\left|\frac{\partial }{\partial x_{i_1}}\frac{\partial }{\partial x_{i_2}}\frac{\partial }{\partial x_{i_3}}\cdots\frac{\partial }{\partial x_{i_m}}u\right|^2\right)\md x\\
    &=\int_{\RR^N}\left(\sum_{i_3,\,\dots,\,i_m=1}^N\left|\Delta\left(\frac{\partial }{\partial x_{i_3}}\cdots\frac{\partial }{\partial x_{i_m}}u\right)\right|^2\right)\md x\\
    &=\int_{\RR^N}\left(\sum_{i_3,\,\dots,\,i_m=1}^N\left|\frac{\partial }{\partial x_{i_3}}\cdots\frac{\partial }{\partial x_{i_m}}(\Delta u)\right|^2\right)\md x\\
    &=\int_{\RR^N}\left|\nabla^{m-2}(\Delta u)\right|^2\md x.
\end{align*}
Therefore, if $m$ is even and $m=2j$,  then $P_m u=\Delta^j u$ and
\begin{align*}
  \int_{\RR^N}\left|\nabla^m u\right|^2\md x=\int_{\RR^N}\left|\nabla^{2j-2}\left(\Delta u\right)\right|^2\md x=\cdots=\int_{\RR^N}\left|\nabla^2\left(\Delta^{j-1}u\right)\right|^2\md x= \int_{\RR^N}\left|\Delta^j u\right|^2\md x.
\end{align*}
If $m$ is odd and $m=2j+1$, then $P_m u=\nabla \Delta^j u$ and
\begin{equation*}
    \int_{\RR^N}\left|\nabla^{m} u\right|^2\md x=\int_{\RR^N}\left|\nabla^{2j-1}\left(\Delta u\right)\right|^2\md x=\cdots=\int_{\RR^N}\left|\nabla^{3}\left(\Delta^{j-1}u\right)\right|^2\md x=\int_{\RR^N}\left|\nabla\left(\Delta^{j}u\right)\right|^2\md x.
\end{equation*}
\end{proof}
Lemma \ref{integral identity 1} describes the relationship between $\lVert\nabla^m u\rVert_2$ and $\lVert P_mu\rVert_2$. To present the relationship between $\lVert|x|\left|\nabla^m u\right|\rVert_2$ and $\lVert|x|\left|P_m u\right|\rVert_2$, we first need the following lemma.
\begin{lemma}
For $N,\,m\in\NN_+$, $m\geq2$ and $u\in C^\oo_0\left(\RR^N\right)$, we have
\begin{align*}
    \int_{\RR^N}|x|^2\left|\nabla^m u\right|^2\md x=\int_{\RR^N}|x|^2\left|\nabla^{m-2}\left(\Delta u\right)\right|^2\md x+2(N-1)\int_{\RR^N}\left|P_{m-1}u\right|^2\md x.
\end{align*}
\label{integral identity 2}
\end{lemma}
\begin{proof}
For the case $m=2$, there holds
\begin{align}
\label{Bochner 1}
    \int_{\RR^N}|x|^2\left|\nabla^2 u\right|^2\md x&=\sum_{k,\,l=1}^N\int_{\RR^N}|x|^2\left(\dfrac{\partial^2 u}{\partial x_l\partial x_k}\right)^2\md x\nonumber\\
    &=-\sum_{k,\,l=1}^N\int_{\RR^N}\dfrac{\partial}{\partial x_l}\left(|x|^2\dfrac{\partial^2 u}{\partial x_l\partial x_k}\right)\dfrac{\partial u }{\partial x_k}\md x\nonumber\\
    &=-\sum_{k,\,l=1}^N\int_{\RR^N}\left(2x_l\dfrac{\partial^2 u}{\partial x_l\partial x_k}+|x|^2\dfrac{\partial^3 u}{\partial x_l\partial x_l\partial x_k}\right)\dfrac{\partial u }{\partial x_k}\md x.
\end{align}
Notice that
\begin{align}
    \label{Bochner 2}
    -\sum_{k,\,l=1}^N\int_{\RR^N}2x_l\dfrac{\partial u}{\pt x_k}\dfrac{\pt^2u}{\pt x_l\pt x_k}\md x&=-\sum_{k,\,l=1}^N\int_{\RR^N}x_l\dfrac{\pt}{\pt x_l}\left(\left(\dfrac{\pt u}{\pt x_k}\right)^2\right)\md x\nonumber\\
    &=\sum_{k,\,l=1}^N\int_{\RR^N}\left(\dfrac{\pt u}{\pt x_k}\right)^2\md x\nonumber\\
    &=N\int_{\RR^N}\left|\nabla u\right|^2\md x.\\
    -\sum_{k,\,l=1}^N\int_{\RR^N}|x|^2\dfrac{\partial^3 u}{\partial x_l\partial x_l\partial x_k}\dfrac{\partial u }{\partial x_k}\md x&=-\sum_{k=1}^N\int_{\RR^N}|x|^2\dfrac{\pt(\Delta u)}{\pt x_k}\dfrac{\pt u}{\pt x_k}\md x\nonumber\\
    &=2\int_{\RR^N}(x\cdot\nabla u)\Delta u\md x+\int_{\RR^N}|x|^2|\Delta u|^2\md x.\label{Bochner 3}
\end{align}
\\
Moreover, we have
\begin{align}
    \label{Bochner 4}
    2\int_{\RR^N}(x\cdot\nabla u)\Delta u\md x&=-2\sum_{k,\,l=1}^N\int_{\RR^N}\dfrac{\pt u}{\pt x_l}\dfrac{\pt }{\pt x_l}\left(x_k\dfrac{\pt u}{\pt x_k}\right)\md x\nonumber\\
    &=-2\int_{\RR^N}|\nabla u|^2\md x-\sum_{k,\,l=1}^N\int_{\RR^N}x_k\dfrac{\pt}{\pt x_k}\left(\left(\dfrac{\pt u}{\pt x_l}\right)^2\right)\md x\nonumber\\
    &=(N-2)\int_{\RR^N}|\nabla u|^2\md x.
\end{align}
\\
Substituting \eqref{Bochner 2}, \eqref{Bochner 3}, \eqref{Bochner 4} into \eqref{Bochner 1}, we obtain the following result:
\begin{equation*}
    \int_{\RR^N}\lvert x\rvert^2\lvert \nabla^2 u\rvert^2\md x=\int_{\RR^N}\lvert x\rvert^2\lvert \Delta u\rvert^2\md x+2(N-1)\int_{\RR^N}\lvert \nabla u\rvert^2\md x.
\end{equation*}
\\
Using the result for $m=2$, we readily obtain the result for $m\geq3$:
\begin{align*}
    \int_{\RR^N}|x|^2\left|\nabla^m u\right|^2\md x&=\sum_{i_1,\,\dots,\,i_{m-2}=1}^N\int_{\RR^N}|x|^2\left|\nabla^2\left(\dfrac{\pt^{m-2}u}{\pt x_{i_1}\cdots\pt x_{i_{m-2}}}\right)\right|^2\md x\\
    &=\sum_{i_1,\,\dots,\,i_{m-2}=1}^N\left[\int_{\RR^N}\lvert x\rvert^2\left\lvert \Delta \left(\dfrac{\pt^{m-2}u}{\pt x_{i_1}\cdots\pt x_{i_{m-2}}}\right)\right\rvert^2\md x\right.\\
    &\hspace{4cm} \left.+2(N-1)\int_{\RR^N}\left\lvert \nabla \left(\dfrac{\pt^{m-2}u}{\pt x_{i_1}\cdots\pt x_{i_{m-2}}}\right)\right\rvert^2\md x\right]\\
    &=\int_{\RR^N}|x|^2\left|\nabla^{m-2}(\Delta u)\right|^2\md x+2(N-1)\int_{\RR^N}\left|\nabla^{m-1}u\right|^2\md x\\
    &=\int_{\RR^N}|x|^2\left|\nabla^{m-2}(\Delta u)\right|^2\md x+2(N-1)\int_{\RR^N}\left|P_{m-1}u\right|^2\md x,
\end{align*}
where the last equality above follows from Lemma \ref{integral identity 1}.
\end{proof}
By repeatedly applying Lemma \ref{integral identity 2}, we obtain the following powerful identity, which will be very important in Section \ref{sec 3}.
\begin{lemma}
\label{integral identity 3}
For $N,\,m\in\NN_+$ and $u\in C^\oo_0\left(\RR^N\right)$, we have
 \begin{align*}
     \int_{\RR^N}|x|^2\left|\nabla^{m} u\right|^2\md x=\int_{\RR^N}|x|^2\left|P_{m}u\right|^2\md x+2(N-1)\left\lfloor \frac{m}{2}\right\rfloor\int_{\RR^N}\left|P_{m-1}u\right|^2\md x,
 \end{align*}
 where $\lfloor m/2\rfloor$ denotes the greatest integer less than or equal to $m/2$.
\end{lemma}
\begin{proof}
The case $m=1$ is trivial. Therefore, we only need to consider $m\geq2$. Let $j=\lfloor m/2\rfloor$. If $m=2j$ for $j\geq 1$, by Lemma \ref{integral identity 2}, there holds
\begin{align*}
    \int_{\RR^N}|x|^2\left|\nabla^m u\right|^2\md x&=\int_{\RR^N}|x|^2\left|\nabla^{2j-2}(\Delta u)\right|^2\md x+2(N-1)\int_{\RR^N}\left|P_{m-1}u\right|^2\md x\\
    &=\int_{\RR^N}|x|^2\left|\nabla^{2j-4}\left(\Delta^2 u\right)\right|^2\md x+4(N-1)\int_{\RR^N}\left|P_{m-1}u\right|^2\md x\\
    &\cdots\\
    &=\int_{\RR^N}|x|^2\left|\Delta^j u\right|^2\md x+2j(N-1)\int_{\RR^N}\left|P_{m-1}u\right|^2\md x\\
    &=\int_{\RR^N}|x|^2\left|P_m u\right|^2\md x+2(N-1)\left\lfloor\dfrac{m}{2}\right\rfloor\int_{\RR^N}\left|P_{m-1}u\right|^2\md x.
\end{align*}
Similarly, if $m=2j+1$ for $j\geq1$, we obtain
\begin{align*}
    \int_{\RR^N}|x|^2\left|\nabla^m u\right|^2\md x&=\int_{\RR^N}|x|^2\left|\nabla^{2j-1}(\Delta u)\right|^2\md x+2(N-1)\int_{\RR^N}\left|P_{m-1}u\right|^2\md x\\
    &=\int_{\RR^N}|x|^2\left|\nabla^{2j-3}\left(\Delta^2 u\right)\right|^2\md x+4(N-1)\int_{\RR^N}\left|P_{m-1}u\right|^2\md x\\
    &\cdots\\
    &=\int_{\RR^N}|x|^2\left|\nabla\Delta^j u\right|^2\md x+2j(N-1)\int_{\RR^N}\left|P_{m-1}u\right|^2\md x\\
    &=\int_{\RR^N}|x|^2\left|P_m u\right|^2\md x+2(N-1)\left\lfloor\dfrac{m}{2}\right\rfloor\int_{\RR^N}\left|P_{m-1}u\right|^2\md x.
\end{align*}
\end{proof}
\section{Proof of Theorem \ref{thm m order HUP}}\label{sec 3}
In this section, we will prove Theorem \ref{thm m order HUP}. The key point is to analyze the following two operators, corresponding to the cases where $m$ is even and where $m$ is odd, respectively:
\begin{align}
    \mathcal{E}_{b,c}(u)&=\Delta u+b x\cdot \nabla u+bc u,\label{def of Ebc}\\
    \mathcal{O}_{b,c}(u)&=\nabla \Delta u+(b\Delta u) x+bc \nabla u,\label{def of Obc}
\end{align}
for $b,\,c\in \RR$ and $b>0$. For any  $u\in C^\oo_0\left(\RR^N\right)$, by integration by parts, we obtain
\begin{align}
\int_{\RR^N}\lvert\mathcal{E}_{b,c}(u)\rvert^2\md x&=\int_{\RR^N}|\Delta u|^2\md x+b^2\int_{\RR^N}(x\cdot \nabla u)^2\md x+b^2c^2\int_{\RR^N}u^2\md x\nonumber\\
&\hspace{1cm}+2b\int_{\RR^N}(x\cdot \nabla u)\Delta u\md x+2bc\int_{\RR^N}u\Delta u\md x+2b^2c\int_{\RR^N}(x\cdot \nabla u)u\md x\nonumber\\
&=\int_{\RR^N}|\Delta u|^2\md x+b^2\int_{\RR^N}(x\cdot \nabla u)^2\md x+b^2c^2\int_{\RR^N}u^2\md x\nonumber\\
&\hspace{1cm}+b(N-2)\int_{\RR^N}|\nabla u|^2\md x-2bc\int_{\RR^N}|\nabla u|^2\md x-b^2cN\int_{\RR^N}u^2\md x\nonumber\\
&=\int_{\RR^N}|\Delta u|^2\md x+b^2\int_{\RR^N}(x\cdot \nabla u)^2\md x+b(N-2-2c)\int_{\RR^N}|\nabla u|^2\md x\nonumber\\
&\hspace{1cm}+b^2\left(c^2-cN\right)\int_{\RR^N}u^2\md x.
    \label{expansion of E}
    \end{align}
    \begin{align}
 \label{expansion of O}
   \int_{\RR^N}\left|\mathcal{O}_{b,c} (u)\right|^2\md x&=\int_{\RR^N}\lvert \nabla \Delta u\rvert^2\md x+b^2\int_{\RR^N}|\Delta u|^2\lvert x\rvert^2\md x+b^2c^2\int_{\RR^N}\lvert \nabla u\rvert^2\md x\nonumber\\
    &\quad +2b\int_{\RR^N}(\nabla \Delta u\cdot x)\Delta u\md x+2bc\int_{\RR^N}\nabla\Delta u\cdot \nabla u\md x+2b^2c\int_{\RR^N}(\nabla u\cdot x)\Delta u\md x\nonumber\\
    &=\int_{\RR^N}\lvert \nabla \Delta u\rvert^2\md x+b^2\int_{\RR^N}|\Delta u|^2\lvert x\rvert^2\md x+b^2c^2\int_{\RR^N}\lvert \nabla u\rvert^2\md x\nonumber\\
    &\quad -b\int_{\RR^N}N |\Delta u|^2\md x-2bc\int_{\RR^N}|\Delta u|^2\md x+b^2c(N-2)\int_{\RR^N}|\nabla u|^2\md x\nonumber\\
    &=\int_{\RR^N}\lvert \nabla\Delta u\rvert^2\md x+b^2\int_{\RR^N}\lvert \Delta u\rvert^2\lvert x\rvert^2\md x+b(-N-2c)\int_{\RR^N}\lvert \Delta u\rvert^2\md x\nonumber\\
    &\quad+b^2\left(c^2+c(N-2)\right)\int_{\RR^N}\lvert \nabla u\rvert^2\md x.
\end{align}
\\
\begin{proof}[Proof of Theroem \ref{thm m order HUP}]
Denote $j=\lfloor m/2\rfloor$.
\vspace{0.5cm}
\\
\textbf{Case 1: }$m=2j$. Let $u\in C^\oo_0\left(\RR^N\right)\backslash\{0\}$ and $v=P_{m-2}u=\Delta^{j-1}u$. Using \eqref{expansion of E} and Lemma \ref{integral identity 3}, we have
\begin{align}
    0&\leq \int_{\RR^N}\lvert\mathcal{E}_{b,c}(v)\rvert^2\md x\nonumber\\
    &\leq\int_{\RR^N}\lvert\Delta v\rvert^2\md x+b^2\int_{\RR^N}\lvert x\rvert^2\lvert\nabla v\rvert^2\md x+b(N-2c-2)\int_{\RR^N}\lvert \nabla v\rvert^2\md x+b^2\left(c^2-cN\right)\int_{\RR^N}v^2\md x\nonumber\\
    &=\int_{\RR^N}|P_m u|^2\md x+b^2\int_{\RR^N}|x|^2\left|P_{m-1}u\right|^2\md x+b(N-2c-2)\int_{\RR^N}\left|P_{m-1}u\right|^2\md x\nonumber\\
    &\hspace{2cm}+b^2\left(c^2-cN\right)\int_{\RR^N}\left|P_{m-2} u\right|^2\md x\nonumber\\
    &=\int_{\RR^N}|P_m u|^2\md x+b^2\int_{\RR^N}|x|^2\left|\nabla^{m-1}u\right|^2\md x+b(N-2c-2)\int_{\RR^N}\left|P_{m-1}u\right|^2\md x\nonumber\\
    &\hspace{2cm}+b^2\left(c^2-cN-(m-2)(N-1)\right)\int_{\RR^N}\left|P_{m-2} u\right|^2\md x.
    \label{linearized even}
\end{align}
Now we take $c_e$ to be the positive root of the equation $c^2-cN-(m-2)(N-1)=0$, namely:
\begin{equation}
\label{def of ce}
    c_e=\frac{1}{2}\left(N+\sqrt{N^2+4(m-2)(N-1)}\right).
\end{equation}
It follows that
\begin{align*}
    b(2c_e-N+2)\int_{\RR^N}\left|P_{m-1}\right|^2\md x\leq \int_{\RR^N}\left|P_m u\right|^2\md x+b^2\int_{\RR^N}|x|^2\left|\nabla^{m-1}u\right|^2\md x,
\end{align*}
for any $b>0$. Since $u\not\equiv0$, we can choose
\begin{equation*}
    b=\left(\int_{\RR^N}\left|P_m u\right|^2\md x\right)^{\frac{1}{2}}\left(\int_{\RR^N}|x|^2\left|\nabla^{m-1}u\right|^2\md x\right)^{-\frac{1}{2}}.
\end{equation*}
Thus, we have
\begin{align*}
    \left(\frac{2c_e-N+2}{2}\right)^2\left(\int_{\RR^N}\left|P_{m-1}\right|^2\md x\right)^2\leq\int_{\RR^N}\left|P_mu\right|^2\md x\int_{\RR^N}|x|^2\left|\nabla^{m-1}u\right|^2\md x.
\end{align*}
\\
\textbf{Case 2: }$m=2j+1$. Let $u\in C^\oo_0\left(\RR^N\right)\backslash\{0\}$ and $v=\Delta^{j-1}u$ such that $\nabla v=P_{m-2}u$. By \eqref{expansion of O} and \eqref{integral identity 3}, we have
\begin{align}
    \int_{\RR^N}\left|\mathcal{O}_{b,c} (v)\right|^2\md x
    &=\int_{\RR^N}\lvert \nabla\Delta v\rvert^2\md x+b^2\int_{\RR^N}|\Delta v|^2\lvert x\rvert^2\md x-b(N+2c)\int_{\RR^N}\lvert \Delta v\rvert^2\md x\nonumber\\
    &\qquad+b^2\left(c^2+c(N-2)\right)\int_{\RR^N}\lvert \nabla v\rvert^2\md x\nonumber\\
    &=\int_{\RR^N}\left\lvert P_m u\right\rvert^2\md x+b^2\int_{\RR^N}\left\lvert \nabla^{m-1}u\right\rvert^2\lvert x\rvert^2\md x-b(N+2c)\int_{\RR^N}\left\lvert P_{m-1}u\right\rvert^2\md x\nonumber\\
    &\qquad+b^2\left(c^2+c(N-2)-(m-1)(N-1)\right)\int_{\RR^N}\left\lvert P_{m-2}u\right\rvert^2\md x.
    \label{linearized odd}
\end{align}
Similarly, we take $c_o$ to be the positive root of the equation $c^2+c(N-2)-(m-1)(N-1)=0$, that is,
\begin{equation}
\label{def of co}
    c_o=\frac{1}{2}\left(2-N+\sqrt{(N-2)^2+4(m-1)(N-1)}\right).
\end{equation}
It follows that
\begin{align*}
    b\left(N+2c_o\right)\int_{\RR^N}\left|P_{m-1}\right|^2\md x\leq \int_{\RR^N}\left|P_m u\right|^2\md x+b^2\int_{\RR^N}|x|^2\left|\nabla^{m-1}u\right|^2\md x,
\end{align*}
for any $b>0$. Since $u\not\equiv0$, we can choose
\begin{equation*}
    b=\left(\int_{\RR^N}\left|P_m u\right|^2\md x\right)^{\frac{1}{2}}\left(\int_{\RR^N}|x|^2\left|\nabla^{m-1}u\right|^2\md x\right)^{-\frac{1}{2}}.
\end{equation*}
Thus, we have
\begin{align*}
    \left(\frac{N+2c_o}{2}\right)^2\left(\int_{\RR^N}\left|P_{m-1}u\right|^2\md x\right)^2\leq\int_{\RR^N}\left|P_mu\right|^2\md x\int_{\RR^N}|x|^2\left|\nabla^{m-1}u\right|^2\md x.
\end{align*}
\\
\indent Finally, in both \textbf{Case\,1} and \textbf{Case\,2}, we have the following relation:
\begin{equation*}
    \mu_m(N)\coloneqq\left(\frac{2+\sqrt{N^2+4(m-2)(N-1)}}{2}\right)^2=\left(\frac{N+2c_o}{2}\right)^2=\left(\frac{2c_e-N+2}{2}\right)^2.
\end{equation*}
Consequently, by density argument, \eqref{eq m order HUP} holds for any $m\geq2$ and $u\in \mathcal{H}^m$.
\end{proof}

Now we have established the $m$-th order Heisenberg uncertainty principle. Next, we will focus on showing that the constant $\mu_m(N)$ is optimal and on proving the existence of the extremal functions. Motivated by \eqref{linearized even} and \eqref{linearized odd}, we will first seek solutions to $\mathcal{E}_{b,c_e}(v)=0$ and $\mathcal{O}_{b,c_o}(v)=0$ in the following section, where $c_e$ and $c_o$  are defined by \eqref{def of ce} and \eqref{def of co}, respectively. For notational simplicity, from now on, we denote $\mathcal{E}_b=\mathcal{E}_{b,c_e}$ and  $\mathcal{O}_b=\mathcal{O}_{b,c_o}$.
\vspace{1cm}
\section{Solutions of  \texorpdfstring{$\mathcal{E}_{b}(v)=0$}{Eb(v)=0} and \texorpdfstring{$\mathcal{O}_{b}(v)=0$}{Ob(v)=0} and Their Asymptotic Properties}\label{sec 4}
In this section, we prove that the equations $\mathcal{E}_{b}(v)=0$ and $\mathcal{O}_{b}(v)=0$ admit radially symmetric solutions for $N\geq2$. And we analyze in detail the corresponding asymptotic behavior of these solutions. For two functions $f$ and $g$, we say that $f(t)$ is of order $g(t)$ as $t\to+\infty$ and write $f(t)=O(g(t))$, if there exists some $t_0>0$ and $C>0$ such that $|f(t)|\leq C|g(t)|$ for all $t\geq t_0$.
\begin{proposition}
\label{expression of v}
For $N\geq 2$, $m\geq2$ and $b>0$, we denote $r=|x|$ and define
\begin{align*}
    v_e(x,\,b)=v_e(r,\,b)=&\mathrm{e}^{-\frac{br^2}{2}}{_1F_1}\left(\dfrac{N-c_e}{2},\,\dfrac{N}{2},\,\dfrac{br^2}{2}\right),\\
    v_o(x,\,b)=v_o(r,\,b)=&-\int_{r}^{+\oo}s\mathrm{e}^{-\frac{bs^2}{2}}{_1F_1}\left(\dfrac{2-c_o}{2},\,\dfrac{N+2}{2},\,\dfrac{bs^2}{2}\right)\md s,
\end{align*}
where $c_e$ is defined by \eqref{def of ce}  and $c_o$ is defined by \eqref{def of co}.  Then $v_e$ and $v_o$ are smooth with respect to $x$ such that $\mathcal{E}_{b}\left(v_e\right)=0$ and $\mathcal{O}_b\left(v_o\right)=0$ hold in $\RR^N$.
\end{proposition}
\begin{proof}
We fix $b>0$. By abuse of notation, we write $v_e(x)=v_e(x,\,b)$ and $v_o(x)=v_o(x,\,b)$ whenever no confusion arises. From Lemma \ref{property of F}, we have
\begin{align*}
    s\mathrm{e}^{-\frac{bs^2}{2}}{_1F_1}\left(\frac{2-c_o}{2},\,\frac{N+2}{2},\,\frac{bs^2}{2}\right)=O\left(s^{1-N-c_o}\right)\text{ as }s\to+\oo,
\end{align*}
where $1-N-c_o<-1$. Thus, $v_o(x)$ is well-defined. Next, we rigorously verify that both $v_e$ and $v_o$ are smooth functions on the whole space $\RR^N$, particularly at the origin $x=0$. Note that a radial function belongs to $C^\oo\left(\RR^N\right)$ if it can be written as a smooth function of $|x|^2$. For $v_e$, since the map $z\mapsto \mathrm{e}^{-z} {_1F_1}\left(\frac{N-c_e}{2},\,\frac{N}{2},\,z\right)$ is an entire function on $\CC$, and $z(x)=\frac{b|x|^2}{2}$ is smooth in $x$, their composition is smooth. Thus $v_e\in C^\oo\left(\RR^N\right)$. For $v_o$, we use the change of variables $\tau=s^2/2$ and the integral can be rewritten as:
\begin{equation*}
    v_o(x)=-\int_{\frac{|x|^2}{2}}^{+\oo}\mathrm{e}^{-{b\tau}}{_1F_1}\left(\dfrac{2-c_o}{2},\,\dfrac{N+2}{2},\,{b\tau}\right)\md \tau.
\end{equation*}
The fundamental theorem of calculus ensures that the map $$t\mapsto -\int_{t}^{+\oo}\mathrm{e}^{-{b\tau}}{_1F_1}\left(\dfrac{2-c_o}{2},\,\dfrac{N+2}{2},\,{b\tau}\right)\md \tau$$ belongs to $C^\oo\left([0,\,+\oo)\right)$. Since $t(x)=\frac{|x|^2}{2}$ is smooth in $x$, it follows that $v_o(x)\in C^\oo\left(\RR^N\right)$. 

Next, we show that $\mathcal{E}_{b}(v_e)=0$. Let $z=\dfrac{br^2}{2}$ and $M_{e}(z)={_1F_1}\left(\dfrac{N-c_e}{2},\,\dfrac{N}{2},\,z\right)$. Then $v_e(r)=\mathrm{e}^{-z}M_e(z)$ and
\begin{align*}
    v_e'(r)&=\frac{\md v_e}{\md z}\frac{\md z}{\md r}=br\mathrm{e}^{-z}\left(M_e'(z)-M_e(z)\right),\\
    v_e''(r)&=b\mathrm{e}^{-z}\left(2zM_e''(z)+(1-4z)M_e'(z)+(2z-1)M_e(z)\right).
\end{align*}
It follows that
\begin{align*}
    \mathcal{E}_b(v_e)&=v_e''(r)+\left(\frac{N-1}{r}+br\right)v_e'(r)+bc_ev_e(r)\\
    &=b\mathrm{e}^{-z}\left[2zM_e''(z)+(N-2z)M_e'(z)+(c_e-N)M_e(z)\right].
\end{align*}
Notice that $M_e$ satisfies the confluent hypergeometric equation
\begin{equation*}
    zM_e''(z)+\left(\frac{N}{2}-z\right)M_e'(z)-\frac{N-c_e}{2} M_e(z)=0.
\end{equation*}
Hence, we have $\mathcal{E}_b\left(v_e\right)=0$. By a similar argument, one can prove that $\mathcal{O}_b\left(v_o\right)=0$. Denote $z=\dfrac{br^2}{2}$ and $M_o(z)={_1F_1}\left(\dfrac{2-c_o}{2},\,\dfrac{N+2}{2},\,z\right)$, then we have
\begin{align*}
    v_o'(r)&=r\mathrm{e}^{-z}M_o(z),\\
    v_o''(r)&=\mathrm{e}^{-z}\left((1-2z)M_o(z)+2zM_o'(z)\right),\\
    v_o''(r)&=br\mathrm{e}^{-z}\left(2zM_o''(z)+(3-4z)M_o'(z)+(2z-3)M_o(z)\right).
\end{align*}
By a direct computation, we have
\begin{align*}
    \mathcal{O}_b\left(v_o\right)&=\frac{x}{r}\left(v_o'''(r)+\dfrac{N-1}{r}v_o''(r)-\dfrac{N-1}{r^2}v_o'(r)+b r\left(v_o''(r)+\dfrac{N-1}{r}v_o'(r)\right)+bc_o v_o'(r)\right)\\
    &=2br\mathrm{e}^{-z}\left[zM_o''(z)+\left(\dfrac{N+2}{2}-z\right)M_o'(z)-\frac{2-c_o}{2}M_o(z)\right].
\end{align*}
As $M_o$ satisfies the confluent hypergeometric equation
\begin{equation*}
    zM_o''(z)+\left(\frac{N+2}{2}-z\right)M_o'(z)-\frac{2-c_o}{2} M_o(z)=0,
\end{equation*}
we have $\mathcal{O}_b\left(v_o\right)=0$.
\end{proof}

In the remainder of this section, we will  divide the analysis into the two cases  $N\geq2$ and $N=1$. Firstly, in the case $N=2$, we seek for some suitable function $\Phi_{b,m}\in C^{\oo}\left(\RR^N\right)$ such that $\Delta^{j-1}\Phi_{b,m}(x)=v_e(x,\,b)$ for $m=2j$ and $\Delta^{j-1}\Phi_{b,m}(x)=v_o(x,\,b)$ for $m=2j+1$. In particular, for $m=2$ and 
$m=3$, it suffices to take $\Phi_{b,2}(x)=v_e(|x|,\,b)$ and $\Phi_{b,3}(x)=v_o(|x|,\,b)$, respectively. For higher-order cases, we introduce the following operator:
\begin{align}
\label{def of I}
    \mathcal{I}[f](r)=\int_{0}^rs^{1-N}\int_0^st^{N-1}f(t)\md t\md s.
\end{align}
For any $f\in C^\oo\left([0,\,+\oo)\right)$, a direct computation shows that $\Delta \left(\mathcal{I}[f](|x|)\right)=f(|x|).$ In order to analyze the asymptotic behavior of the function after applying the operator $\mathcal{I}$, we define the following set.
\begin{definition}
For $a\in\RR$ and $\sigma\in \NN_+\cup\{0\}$, define
\begin{equation*}
    \mathcal{W}(a,\,\sigma)=\left\{f\in C^\oo\left((0,\,+\oo)\right):\,\text{for all }l\in\NN_+\cup\{0\},\,f^{(l)}(r)=O\left(r^{-a-l}\ln^\sigma r\right)\text{ as }r\to+\oo\right\}.
\end{equation*}
\end{definition}
\begin{remark}
By the definition of the set $\mathcal{W}(a,\,\sigma)$, the following conclusions can be easily verified:
\begin{enumerate}
    \item For any $\sigma_1,\,\sigma_2\in\NN_+\cup\{0\}$, if $a_1>a_2$, then $\mathcal{W}\left(a_1,\,\sigma_1\right)\subset\mathcal{W}\left(a_2,\,\sigma_2\right)$.
\item For any $a\in\RR$, if $\sigma_1\leq\sigma_2$, then $\mathcal{W}\left(a,\,\sigma_1\right)\subset\mathcal{W}\left(a,\,\sigma_2\right)$.
\item If $f(r)\in\mathcal{W}(a,\,\sigma)$ and $k\in\RR$, then we have $r^kf(r)\in \mathcal{W}(a-k,\,\sigma)$.
\end{enumerate}
\end{remark}
\begin{lemma}
Let $N\geq 2$, $a\in\RR$ and $\sigma\in\NN_+\cup\{0\}$. Define
\begin{equation*}
    H(r)=\begin{cases}
    r^{2-N},\qquad &\text{if }N\geq3,\\
    \ln r,&\text{if }N=2.
    \end{cases}
\end{equation*}
For $f\in\mathcal{W}(a,\,\sigma)\cap C([0,\,+\oo))$ and large $
r>0$, there holds:
\begin{align*}
\mathcal{I}[f](r)=C_1+C_2H(r)+R(r),
\end{align*}
where $\mathcal{I}[f]$ is defined by \eqref{def of I}, $C_1$ and $C_2$ are constants depending on $f$ and do not depend on $r$. Moreover, the remainder $R\in \mathcal{W}(a-2,\,\sigma+2)$.
\label{estimate of I}
\end{lemma}
\begin{proof}
Since $f\in\mathcal{W}(a,\,\sigma)$, there exists some $r_0>0$ and $C>0$ such that $|f(r)|\leq Ct^{-a}\ln^\sigma(r)$ for $r\geq r_0$. Denote $\mathcal{J}[f](s)=s^{1-N}\int_0^s t^{N-1}f(t)\md t$. We shall prove this lemma by considering the following three cases.
\\
\textbf{Case 1: }$a>N$. As $N-a<0$, the integral $\int_0^{+\oo} t^{N-1}f(t)\md t$ converges. It follows that
\begin{align}
    \mathcal{J}[f](s)&=s^{1-N}\left[\int_0^\oo t^{N-1}f(t)\md t-\int_s^\oo t^{N-1}f(t)\md t\right]\nonumber\\
    &=Cs^{1-N}-s^{1-N}\int_s^{+\oo}t^{N-1}f(t)\md t.
    \label{Jf in case 1}
\end{align}
Note that
\begin{align*}
    \left|-\int_s^{+\oo}t^{N-1}f(t)\md t\right|\leq C'\int_s^{+\oo}t^{N-1-a}\ln^\sigma(t)\md t\leq C' s^{N-a}\ln^\sigma(s)\qquad\text{for large } s.
\end{align*}
It follows that 
\begin{align*}
&\frac{\md }{\md s}\left(-\int_s^{+\oo}t^{N-1}f(t)\md t\right)=s^{N-1}f(s)\in\mathcal{W}(a-N+1,\,\sigma),\\
& -\int_s^{+\oo}t^{N-1}f(t)\md t\in\mathcal{W}(a-N,\,\sigma),\\
&R_1(s)\coloneqq -s^{1-N}\int_s^{+\oo}t^{N-1}f(t)\md t\in\mathcal{W}(a-1,\,\sigma).
\end{align*}
Then, by \eqref{Jf in case 1}, for large $r>0$, we have
\begin{align}
    \mathcal{I}[f](r)&=\int_0^1\mathcal{J}[f](s)\md s+C\int_1^r s^{1-N}\md s+\int_1^{+\oo}R_1(s)\md s-\int_r^{+\oo}R_1(s)\md s\nonumber\\
    &=C_1+C_2 H(r)-\int_r^{+\oo}R_1(s)\md s.
    \label{If in case 1}
\end{align}
Similarly, $R(r)=-\int_r^{+\oo}R_1(s)\md s$ satisfies:
\begin{align*}
    &\left|-\int_r^{+\oo}R_1(s)\md s\right|\leq C'\int_r^{+\oo}s^{1-a}\ln^\sigma(s)\md s\leq C'r^{2-a}\ln^\sigma(r),\\
    &R'(r)=R_1(r)\in \mathcal{W}(a-1,\,\sigma).
\end{align*}
Hence, $R(r)\in\mathcal{W}(a-2,\,\sigma)\subset\mathcal{W}(a-2,\,\sigma+2)$.
\\
\textbf{Case 2: }$a=N$. For large $s>r_0$, there holds:
\begin{align*}
    \left|\int_{r_0}^st^{N-1}f(t)\md t\right|\leq C\int_{r_0}^st^{N-1}\cdot t^{-N}\ln^{\sigma}(t)\md t\leq C\ln^{\sigma+1}(s).
\end{align*}
Notice that
\begin{align*}
    \frac{\md}{\md s}\left(\int_{r_0}^st^{N-1}f(t)\md t\right)&=s^{N-1}f(s)\in\mathcal{W}(1,\,\sigma)\subset\mathcal{W}(1,\,\sigma+1).
\end{align*}
Hence, we have
\begin{align*}
    \int_{0}^st^{N-1}f(t)\md t&=\int_{0}^{r_0}t^{N-1}f(t)\md t+\int_{r_0}^st^{N-1}f(t)\md t\in\mathcal{W}(0,\,\sigma+1)
\end{align*}
Then $\mathcal{J}[f](s)=s^{1-N}\int_0^s t^{N-1}f(t)\md t\in\mathcal{W}(N-1,\,\sigma+1)$. There exist some $r_1>0$ and $C>0$ such that
\begin{equation*}
    \left|\mathcal{J}[f](r)\right|\leq Cr^{1-N}\ln^{\sigma+1}(r)\qquad\text{for }r\geq r_1.
\end{equation*}

If $N>2$, then $\int_0^{+\oo}\mathcal{J}[f](s)\md s$ converges and we denote its value by $C_1$. Therefore
\begin{align*}
    \mathcal{I}[f](r)=C_1+R(r)\qquad\text{for }R(r)=-\int_{r}^{+\oo}\mathcal{J}[f](s)\md s.
\end{align*}
For sufficiently large $r>r_1$, we obtain
\begin{align*}
    &|R(r)|\leq C\int_{r}^{+\oo}s^{1-N}\ln^{\sigma+1}(s)\md s\leq C r^{2-N}\ln^{\sigma+1}(r),\\
    &R'(r)=\mathcal{J}[f](s)=s^{1-N}\int_0^s t^{N-1}f(t)\md t\in\mathcal{W}(N-1,\,\sigma+1).
\end{align*}
That is, $R(r)\in\mathcal{W}(N-2,\,\sigma+1)\subset\mathcal{W}(a-2,\,\sigma+2)$.

If $N=2$, for sufficiently large $r>r_1$, we have
\begin{align*}
    \left|\mathcal{I}[f](r)\right|&\leq \left|\int_0^{r_1}\mathcal{J}[f](s)\md s\right|+C\int_{r_1}^rs^{-1}\ln^{\sigma+1}(s)\md s\leq C'\ln^{\sigma+2}(r),\\
    \mathcal{I}[f]'(r)&=\mathcal{J}[f](r)\in\mathcal{W}(1,\,\sigma+1)\subset\mathcal{W}(1,\,\sigma+2).
\end{align*}
Consequently, $\mathcal{I}[f](r)\in\mathcal{W}(0,\,\sigma+2)=\mathcal{W}(a-2,\,\sigma+2)$.
\\
\textbf{Case 3: }$a<N$. Since $f\in\mathcal{W}(a,\,\sigma)$, for large $s$, there holds:
\begin{align*}
    \left|\int_0^st^{N-1}f(t)\md t\right|&\leq \left|\int_0^{r_0}t^{N-1}f(t)\md t\right|+C\int_{r_0}^st^{N-1-a}\ln^{\sigma}(t)\md t\\
    &\leq C's^{N-a}\ln^{\sigma}(s).
 \end{align*}
It follows that
\begin{align*}
&\frac{\md}{\md s}\left(\int_0^st^{N-1}f(t)\md t\right)=s^{N-1}f(s)\in\mathcal{W}(a-N+1,\,\sigma),\\
&\int_0^st^{N-1}f(t)\md t\in \mathcal{W}(a-N,\,\sigma).
\end{align*}
Then we have
\begin{align*}
    &\mathcal{J}[f](s)=s^{1-N}\int_0^st^{N-1}f(t)\md t\md s\in\mathcal{W}(a-1,\,\sigma).
\end{align*}
\\
Thus, there exist some $r_2>0$ and $C>0$ such that
\begin{equation*}
    \left|\mathcal{J}[f](r)\right|\leq Cr^{1-a}\ln^{\sigma}(r)\qquad\text{for }r\geq r_2.
\end{equation*}

If $2<a<N$, then $a-1>1$,  $\int_0^{+\oo}\mathcal{J}[f](s)\md s$ converges and we denote its value by $C_1$. Therefore, we arrive at
\begin{align*}
    \mathcal{I}[f](r)=C_1+R(r)\qquad\text{for }R(r)=-\int_{r}^{+\oo}\mathcal{J}[f](s)\md s.
\end{align*}
For sufficiently large $r>r_2$, we obtain
\begin{align*}
    &|R(r)|\leq C\int_{r}^{+\oo}s^{1-a}\ln^{\sigma}(s)\md s\leq C r^{2-a}\ln^{\sigma}(r),\\
    &R'(r)=\mathcal{J}[f](s)=s^{1-N}\int_0^s t^{N-1}f(t)\md t\in\mathcal{W}(a-1,\,\sigma).
\end{align*}
That is, $R(r)\in\mathcal{W}(a-2,\,\sigma)\subset\mathcal{W}(a-2,\,\sigma+2)$.

If $a=2<N$ or $a<2$, for sufficiently large $r>r_2$, we have
\begin{align*}
    \left|\mathcal{I}[f](r)\right|&\leq \left|\int_0^{r_2}\mathcal{J}[f](s)\md s\right|+C\int_{r_2}^rs^{1-a}\ln^{\sigma}(s)\md s\leq C'r^{2-a}\ln^{\sigma+1}(r),\\
    \mathcal{I}[f]'(r)&=\mathcal{J}[f](r)\in \mathcal{W}(a-1,\,\sigma)\subset\mathcal{W}(a-1,\,\sigma+1).
\end{align*}
This implies that $\mathcal{I}[f](r)\in\mathcal{W}(a-2,\,\sigma+1)\subset\mathcal{W}(a-2,\,\sigma+2)$.
\end{proof}
\vspace{0.5cm}
Using the explicit forms of $v_e$ and $v_o$  given in Proposition \ref{expression of v},  we obtain the following asymptotic estimates.
\begin{lemma}
\label{asymptotic of v}
For $N\geq2$, $m\geq2$ and $b>0$, we have
\begin{align*}
v_e(r,\,b)\in\mathcal{W}\left(c_e,\,0\right),\qquad v_o(r,\,b)\in\mathcal{W}\left(c_o+N-2,\,0\right).
\end{align*}
\end{lemma}
\begin{proof}
By Kummer’s transformation formula given in Lemma \ref{property of F}, we have
\begin{align*}
v_e(r,\,b)={_1F_1}\left(\dfrac{c_e}{2},\,\dfrac{N}{2},\,-z\right)\qquad\text{for } z=\dfrac{br^2}{2}.
\end{align*}
Together with the differentiation formula from Lemma \ref{property of F} and Fa\`a di Bruno's formula for the derivative of composite functions (see \cite{faadibruno1857}), for $k\in\NN_+\cup\{0\}$, there holds
\begin{align*}
    \frac{\md^k}{\md^k r}v_e(r,\,b)&=\sum_{l=0}^{\lfloor k/2\rfloor}C(k,\,l)r^{k-2l}{_1F_1}\left(\dfrac{c_e}{2}+k-l,\,\dfrac{N}{2}+k-l,-z\right)\\
    &=\sum_{l=0}^{\lfloor k/2\rfloor}C(k,\,l)r^{k-2l}\mathrm{e}^{-z}{_1F_1}\left(\dfrac{N-c_e}{2},\,\dfrac{N}{2}+k-l,z\right)\\
    &=O\left(r^{-c_e-k}\right)\text{  as }r\to+\oo,
\end{align*}
where the last equality follows from the asymptotic estimate given in Lemma \ref{property of F}. This implies that $v_e(r,\,b)\in\mathcal{W}\left(c_e,\,0\right)$. Similarly, we have
\begin{align*}
\frac{\md}{\md r}v_o(r,\,b)&=r\mathrm{e}^{-z}{_1F_1}\left(\dfrac{2-c_o}{2},\,\dfrac{N+2}{2},\,z\right)=r\cdot{_1F_1}\left(\dfrac{N+c_o}{2},\,\dfrac{N+2}{2},\,-z\right),\\
\frac{\md^{k+1}}{\md^{k+1} r}v_o(r,\,b)&=r\sum_{l=0}^{\lfloor k/2\rfloor}C(k,\,l)r^{k-2l}{_1F_1}\left(\dfrac{N+c_o}{2}+k-l,\,\dfrac{N+2}{2}+k-l,-z\right)\\
&\qquad+k\sum_{l=0}^{\lfloor (k-1)/2\rfloor}C(k,\,l)r^{k-1-2l}{_1F_1}\left(\dfrac{N+c_o}{2}+k-1-l,\,\dfrac{N+2}{2}+k-1-l,-z\right)\\
&=O\left(r^{1-N-c_o-k}\right) \text{  as }r\to+\oo.
\end{align*}
Hence, we obtain $v_o(r,\,b)\in\mathcal{W}\left(N+c_o-2,\,0\right)$.
\end{proof}

\begin{proposition}
\label{expression of Ik}
Let $N\geq2,\,j\in \NN_+$ and $j>1$. Denote by  $\mathcal{I}^k$ the $k$-th iteration of $\mathcal{I}$ for $k\in\NN_+$, i.e. $\mathcal{I}^k[\cdot]=\mathcal{I}\left[\mathcal{I}^{k-1}[\cdot]\right]$. For  large $r>0$, we have
\begin{align*}
    \mathcal{I}^k\left[v_e\right](r)&=\sum_{i=1}^{n_k}f_i(r)+Q_{2k-2}(r)\qquad\text{   for }m=2j,\\
    \mathcal{I}^{k}\left[v_o\right](r)&=\sum_{i=1}^{\tilde n_k}\tilde{f}_i(r)+\tilde{Q}_{2k-2}(r)\qquad\text{   for }m=2j+1,
\end{align*}
where $f_i\in \mathcal{W}\left(N-2k,\,\sigma_i\right)$  and $\tilde{f}_i\in \mathcal{W}\left(N-2k,\,\tilde{\sigma}_i\right)$ with $\sigma_i,\,\tilde{\sigma}_i\in\NN_+\cup\{0\}$. $Q_{2k-2}$ and $\tilde{Q}_{2k-2}$ are polynomials of degree at most $(2k-2)$  containing only even powers.
\end{proposition}
\begin{proof}
We proceed by induction. For the base case $k=1$, by Lemma \ref{asymptotic of v}, we have $v_e\in\mathcal{W}\left(c_e,\,0\right)$ and $v_o\in\mathcal{W}\left(N+c_o-2,\,0\right)$. From Lemma \ref{estimate of I}, for large $r>0$, there holds:
\begin{align*}
    \mathcal{I}\left[v_e\right](r)&=C_1+C_2H(r)+R_1(r),\\
    \mathcal{I}\left[v_o\right](r)&=\tilde{C}_1+\tilde{C}_2H(r)+\tilde{R}_1(r),
\end{align*}
where $H\in\mathcal{W}\left(N-2,\,1\right),\,R_1\in\mathcal{W}\left(c_e-2,\,2\right)$ and $\tilde{R}_1\in\mathcal{W}\left(N+c_o-4,\,2\right)$. Since $c_e\geq N$ for $m=2j$ and $N+c_o-2\geq N$ for $m=2j+1$, we have $R_1\in\mathcal{W}\left(N-2,\,2\right)$ and $\tilde{R}_1\in\mathcal{W}(N-2,\,2)$. Thus, the case $k=1$ is proved.

Now, assuming the statement holds for $1\leq k\leq j-2$, we prove it for $k+1$. In the case $m=2j$, suppose that for some $r_0>0$, we have $\mathcal{I}^k\left[v_e\right](r)=\sum_{i=1}^{n_k} f_i(r)+Q_{2k-2}(r)$ for $r\geq r_0$. It follows that
\begin{align*}
    \mathcal{I}^{k+1}\left[v_e\right](r)&=\int_{0}^{r_0}s^{1-N}\int_0^s t^{N-1}\mathcal{I}^k\left[v_e\right](t)\md t\md s\\
    &\qquad+\int_{r_0}^rs^{1-N}\left[\int_0^{r_0}t^{N-1}\mathcal{I}^k\left[v_e\right](t)\md t+\int_{r_0}^st^{N-1}\mathcal{I}^k\left[v_e\right](t)\md t\right]\md s\\
    &=C'+\int_{r_0}^rs^{1-N}\left[C+\int_{r_0}^st^{N-1}\mathcal{I}^k\left[v_e\right](t)\md t\right]\md s\\
    &=C'+CH(r)+\int_{r_0}^rs^{1-N}\int_{r_0}^st^{N-1}\mathcal{I}^k\left[v_e\right](t)\md t\md s.
\end{align*}
Now substitute $\mathcal{I}^k\left[v_e\right](r)=\sum_{i=1}^{n_k} f_i(r)+Q_{2k-2}(r)$ into the above integral. Notice that $Q_{2k-2}$ is a polynomial of degree at most $(2k-2)$  containing only even powers, then $$\int_{r_0}^rs^{1-N}\int_{r_0}^st^{N-1} Q_{2k-2}(t)\md t\md s$$ is the sum of a polynomial of degree at most $2k$  containing only even powers and a constant multiple of $H(r)$. Moreover, for $f_i\in\mathcal{W}\left(N-2k,\,\sigma_i\right)$, repeating the proof of Lemma  \ref{estimate of I}, except that the lower limits of the two integrals in $\mathcal{I}$ are replaced by $r_0$, we obtain
\begin{align*}
    \int_{r_0}^rs^{1-N}\int_{r_0}^st^{N-1}f_i(t)\md t\md s
    &=C'+CH(r)+R_{f_i}(r),
\end{align*}
for large $r>0$, where $R_{f_i}\in \mathcal{W}\left(N-2k-2,\,\sigma_i+2\right)$. Therefore, for sufficiently large $r>0$, we obtain
\begin{align*}
    \mathcal{I}^{k+1}\left[v_e\right](r)&=\sum_{i=1}^{n_{k+1}}\bar f_i(r)+Q_{2k}(r),
\end{align*}
for $\bar f_i\in\mathcal{W}(N-2(k+1),\,\sigma_i+2)$ and $Q_{2k}$ is a polynomial of degree at most $2k$  containing only even powers. For the case $m=2j+1$, the computation for $\mathcal{I}^{k+1}\left[v_o\right]$ is similar to the above, so we omit it here. Thus we complete the proof of this proposition.
\end{proof}
\begin{remark}
Let $r=|x|$, then $\mathcal{I}^k\left[v_e\right](r)$ and $\mathcal{I}^k\left[v_o\right](r)$ are smooth functions in $\RR^N$. In particular, since $v_e(r,\,b)$ and $v_o(r,\,b)$ can be expanded as power series containing only even powers in the neighborhood of 0, $\mathcal{I}^k\left[v_e\right](r)$ and $\mathcal{I}^k\left[v_o\right](r)$ also admit expansions in some neighborhood of 0 containing only even powers, which implies that $\mathcal{I}^k\left[v_e\right](r)$ and $\mathcal{I}^k\left[v_o\right](r)$ are smooth at the origin.
\end{remark}
\begin{proposition}
\label{estimate of Phi radially}
For $N\geq2$, $m\geq2$ and $b>0$, we denote  $j=\lfloor m/2\rfloor$ and $r=|x|$. Define
\begin{align}
\label{def of extremal}
    \Phi_{b,m}(x)=\Phi_{b,m}(r)=\begin{cases}
    v_e(r,\,b),&\text{ if }m=2,\\
    v_o(r,\,b),&\text{ if }m=3,\\
    \mathcal{I}^{j-1}\left[v_e\right](r)-Q_{2j-4}(r),&\text{ if } m=2j\quad\text{ for }j>1,\\
    \mathcal{I}^{j-1}\left[v_o\right](r)-\tilde{Q}_{2j-4}(r),&\text{ if } m=2j+1\quad\text{ for }j>1,
    \end{cases}
\end{align}
where $Q_{2j-4}$ and $\tilde{Q}_{2j-4}$ are polynomials of degree at most $(2j-4)$  containing only even powers given in Proposition \ref{expression of Ik}. Then the following equation holds on $\RR^N$:
\begin{equation}
\label{relationship between Phi and v}
    \Delta^{j-1}\Phi_{b,m}=\begin{cases}
    v_e\,,\qquad&\text{if $m=2j$},\\
    v_o\,,&\text{if $m=2j+1$}.
    \end{cases}
\end{equation}
Furthermore, for all $m\geq2$, we have $\Phi_{b,m}(r)\in\mathcal{W}(N-m+2,\,\sigma)$ for some $\sigma\in\NN_+\cup\{0\}$ depending on $m$.
\end{proposition}
\begin{proof}
Notice that
\begin{equation*}
    \Delta \left(\mathcal{I}\left[v_e\right]\right)= \frac{\md^2 \mathcal{I}\left[v_e\right]}{\md^2 r}(r)+\frac{N-1}{r}\frac{\md \mathcal{I}\left[v_e\right]}{\md r}(r)=v_e.
\end{equation*}
Additionally, for $j>1$, $\Delta^{j-1}Q_{2j-4}=0$ since the degree of  $Q_{2j-4}$ is less than $2j-2$. Therefore, we obtain $\Delta^{j-1}\Phi_{b,m}=v_e$ for $m=2j$. By a similar argument, we have $\Delta^{j-1}\Phi_{b,m}=v_o$ for $m=2j+1$.

As for the asymptotic properties of $\Phi_{b,m}(r)$, by Lemma \ref{asymptotic of v}, we obtain:
\begin{align*}
    \Phi_{b,2}(r)&=v_e(r,\,b)\in \mathcal{W}\left(c_e,\,0\right)=\mathcal{W}\left(N,\,0\right),\\
    \Phi_{b,3}(r)&=v_o\left(r,\,b\right)\in \mathcal{W}\left(c_o+N-2,\,0\right)\subset \mathcal{W}\left(N-1,\,0\right).
\end{align*}
Here we used the fact that $c_e=N$ and $c_o\in(1,2)$ when $j=1$ and $N\geq2$. In addition, by Lemma \ref{expression of Ik}, for sufficiently large $r>0$, we have the following result for $j>1$:
\begin{align*}
    \Phi_{b,m}(r)=\begin{cases}
    \sum_{i=1}\limits^{n}f_i(r),\qquad&\text{if }m=2j,\\
    \sum_{i=1}\limits^{\tilde n}\tilde{f}_i(r),&\text{if }m=2j+1,
    \end{cases}
\end{align*}
where $f_i\in\mathcal{W}\left(N-2j+2,\,\sigma_i\right)$ and $\tilde{f}_i\in\mathcal{W}\left(N-2j+2,\,\tilde{\sigma}_i\right)$ for some $\sigma_i,\,\tilde{\sigma}_i\in\NN_+\cup\{0\}$. Thus, we can choose $\sigma=\max\{\sigma_1,\,\dots,\,\sigma_n\}\text{ for }m=2j\text{ and }\sigma= \{\tilde{\sigma}_1,\,\dots,\,\tilde{\sigma}_{\tilde{n}}\}$ for $m=2j+1$ such that
\begin{align*}
\Phi_{b,m}(r)\in\mathcal{W}(N-2j+2,\,\sigma)\subset \mathcal{W}\left(N-m+2,\,\sigma\right).
\end{align*}
\end{proof}

As an ending to this section, we provide  the analysis for the case  $N=1$. We shall prove that there exists
 $\Phi_{b,m}\in C^\oo\left(\RR\right)$ such that $\Phi_{b,m}^{(m-2)}(x)=\mathrm{e}^{-bx^2}$. Define
 \begin{equation*}
     \mathcal{L}_k(x)=\frac{1}{(k-1)!}\int_0^x(x-t)^{k-1} \mathrm{e}^{-bt^2}\md t\qquad\text{for }k\in\NN_+.
 \end{equation*}
 By a direct computation, we readily obtain that
 \begin{align}
 \label{derivative of Lk}
     \frac{\md}{\md x}\mathcal{L}_{k+1}(x)=\mathcal{L}_k(x).
 \end{align}
 It follows that $\mathcal{L}_k\in C^\oo(\RR)$. Additionally, we have the following estimate.
 \begin{lemma}
 \label{estimate of Lk}
 For any $k\in\NN_+$, there holds $\mathcal{L}_k(x)=O\left(|x|^{k-1}\right)$ as $|x|\to+\oo$.
 \end{lemma}
 \begin{proof}
 Notice that
 \begin{align*}
     \mathcal{L}_k(x)&=\frac{1}{(k-1)!}\int_{0}^x(x-t)^{k-1}\mathrm{e}^{-bt^2}\md t\\
     &=\frac{1}{(k-1)!}\int_{0}^{-x}(x+t)^{k-1}\mathrm{e}^{-bt^2}(-1)\md t\\
     &=\frac{(-1)^k}{(k-1)!}\int_{0}^{-x}(-x-t)^{k-1}\mathrm{e}^{-bt^2}\md t\\
     &=(-1)^k \mathcal{L}_k(-x).
 \end{align*}
 Therefore, we only need to consider the case $x\to+\oo$. Expanding $(x-t)^{k-1}$, we obtain
 \begin{align}
 \label{expansion of Lk}
     \mathcal{L}_k(x)&=\sum_{l=0}^{k-1}C(k,\,l)x^{k-1-l}\int_0^x t^l\mathrm{e}^{-bt^2}\md t\nonumber\\
     &=\sum_{l=0}^{k-1}C(k,\,l)x^{k-1-l}\int_0^{+\oo} t^l\mathrm{e}^{-bt^2}\md t-\sum_{l=0}^{k-1}C(k,\,l)x^{k-1-l}\int_x^{+\oo} t^l\mathrm{e}^{-bt^2}\md t.
 \end{align}
 Next, we will prove $\int_x^{+\oo} t^l\mathrm{e}^{-bt^2}\md t=O\left(x^{l-1}\mathrm{e}^{-bx^2}\right)$ as $x\to+\oo$ by induction. For $l=0$ and $l=1$, we have
 \begin{align*}
     \int_x^{+\oo}\mathrm{e}^{-bt^2}\md t&\leq\frac{1}{2bx}\int_x^{+\oo}2bt\cdot\mathrm{e}^{-bt^2}\md t=\frac{1}{2bx}\mathrm{e}^{-bx^2},\\
     \int_x^{+\oo}t\mathrm{e}^{-bt^2}\md t&=-\frac{1}{2b}\int_x^{+\oo}\md \left(\mathrm{e}^{-bt^2}\right)=\frac{1}{2b}\mathrm{e}^{-bx^2}.
 \end{align*}
 Combining the fact that
 \begin{align*}
     \int_{x}^{+\oo}t^l\mathrm{e}^{-bt^2}\md t&=-\frac{1}{2b}\int_x^{+\oo}t^{l-1}\md \left(\mathrm{e}^{-bt^2}\right)=\frac{1}{2b}x^{l-1}\mathrm{e}^{-bx^2}+\frac{l-1}{2b}\int_x^{+\oo}t^{l-2}\mathrm{e}^{-bt^2}\md t,
 \end{align*}
 we conclude that $\int_x^{+\oo} t^l\mathrm{e}^{-bt^2}\md t=O\left(x^{l-1}\mathrm{e}^{-bx^2}\right)$. Substituting this estimate into \eqref{expansion of Lk}, we obtain $\mathcal{L}_k(x)=O\left(|x|^{k-1}\right)$.
 \end{proof}

Hence, for the special case $N=1$, we set  $\Phi_{b,m}$ as
\begin{align}
\label{def of Phi when N=1}
    \Phi_{b,m}(x)=\begin{cases}
    \mathrm{e}^{-bx^2},\qquad&\text{if }m=2,\\
    \mathcal{L}_{m-2}(x),&\text{if }m\geq3.
    \end{cases}
\end{align}
It follows that $\Phi_{b,m}^{(m-2)}(x)=\mathrm{e}^{-bx^2}$ and $\Phi_{b,m}^{(k)}(x)=O\left(|x|^{m-3-k}\right)$ as $|x|\to+\oo$ for all $k\in\NN_+\cup\{0\}$.
\section{Proofs of Theorem \ref{extremal}, Corollary \ref{cor m order HUP} and Corollary \ref{cor Mazya}}\label{sec 5}
In this section, we first prove Theorem \ref{extremal}, namely that the constant $\mu_m(N)$ in the $m$-th order uncertainty principle \eqref{eq m order HUP} is optimal, and we establish the existence of extremal functions. Then we prove Corollary \ref{cor m order HUP} and  Corollary \ref{cor Mazya}. We introduce the multi-index notation. For $\beta=\left(\beta_1,\,\dots,\,\beta_N\right)$ with each $\beta_i\in \NN_+\cup\{0\}$, define
\begin{equation*}
    |\beta|=\sum_{i=1}^N\beta_i,\qquad D^\beta f=\frac{\pt^{|\beta|}f}{\pt x^{\beta_1}_{1}\cdots\pt x^{\beta_N}_N}\qquad\text{for }f\in C^\oo\left(\RR^N\right).
\end{equation*}
\begin{lemma}
Let $N\geq2$, $m\geq2$ and $b>0$. Then, for any multi-index $\beta$ and $\Phi_{b,m}(x)$ given by \eqref{def of extremal}, as $r=|x|\to+\oo$, we have
\begin{equation*}
    D^\beta\Phi_{b,m}(x)=O\left(r^{m-2-N-|\beta|}\ln^{\sigma}(r)\right),
\end{equation*}
where $\sigma\in\NN_+\cup\{0\}$ depending on $m$.
\label{estimate of Phi}
\end{lemma}
\begin{proof}
When $|\beta|=0$, the estimate of $\Phi_{b,m}$ follows directly from  Proposition  \ref{estimate of Phi radially}. When $|\beta|\geq1$, for radially symmetric function $\Phi_{b,m}$, it is easy to prove by induction that
\begin{align}
    D^\beta\Phi_{b,m}(x)=\sum_{l=1}^{|\beta|}f_{\beta,l}(x)\Phi_{b,m}^{(l)}(r).
    \label{derivative of radial function}
\end{align}
In addition, $f_{\beta,l}$ satisfies
\begin{equation*}
    |f_{\beta,l}(x)|\leq C(\beta,\,l)|x|^{l-|\beta|}.
\end{equation*}
The proposition is proved by the fact  that $\Phi_{b,m}(r)\in\mathcal{W}(N-m+2,\,\sigma)$ for some $\sigma\in\NN_+\cup\{0\}$.
\end{proof}

Next, we consider the truncation of $\Phi_{b,m}$. Let $\eta\in C^\oo_0\left(\RR;\,[0,\,1]\right)$ be a cut-off function such that $\eta(x)=1$ for $|x|\leq1$ and $\eta(x)=0$ for $|x|\geq2$. Define 
\begin{equation}
    \eta_L(x)=\eta(|x|/L)\qquad \text{for }L>0,
    \label{def of eta L}
\end{equation}
then we have
\begin{equation}
\label{estimate of eta}
    \lvert D^\beta \eta_L(x)\rvert\leq {C}{L^{-|\beta|}},
\end{equation}
for some constant $C$ which is independent of $L$.
\begin{lemma}
\label{cut off of Phi}
Let $N\geq2$, $m\geq2$ and $b>0$. For $L>0$, we denote $u_L(x)=\eta_{L}(x)\Phi_{b,m}(x)$, where $\Phi_{b,m}$ is given by \eqref{def of extremal} and $\eta_L$ is given by \eqref{def of eta L}. Define $j=\lfloor m/2\rfloor$ and $v_L=\Delta^{j-1}u_L$. If $m$ is even, then
\begin{equation*}
    \lim_{L\to+\oo}\int_{\RR^N}\lvert\mathcal{E}_b\left( v_L\right)\rvert^2\md x=0.
\end{equation*}
If $m$ is odd, then
\begin{equation*}
    \lim_{L\to+\oo}\int_{\RR^N}\lvert\mathcal{O}_b \left(v_L\right)\rvert^2\md x=0.
\end{equation*}
\end{lemma}
\begin{proof}
A direct computation shows that
\begin{align}
    v_L=\begin{cases}
    u_L,&\text{if }j=1,\\
    \eta_L\Delta^{j-1}\Phi_{b,m}+w_L,\quad&\text{if }j>1,
    \end{cases}
        \label{expression of vL}
\end{align}
where $w_L$ is a linear combination of products in the form $D^{\beta_1}\eta_L\cdot D^{\beta_2}\Phi_{b,m}$ for $|\beta_1|+|\beta_2|=2j-2$ and $|\beta_1|\geq1$.
\\
\textbf{Case 1: }$m$ is even. In this case,  we have $\Delta^{j-1}\Phi_{b,m}(x)=v_e(x,\,b)$. By Proposition \ref{expression of v} and Lemma \ref{asymptotic of v}, we obtain
\begin{align*}
    \mathcal{E}_b\left(\eta_L v_e\right)&=\eta_L\mathcal{E}_b\left(v_e\right)+v_e\Delta\eta_L+2\nabla v_e\nabla\eta_L+bv_e\left(x\cdot\nabla\eta_L\right)\\
    &=v_e\Delta\eta_L+2\nabla v_e\nabla\eta_L+bv_e\left(x\cdot\nabla\eta_L\right)\\
    &=O\left(|x|^{-c_e}\right),\qquad\text{as }|x|\to+\oo.
\end{align*}
It follows that
\begin{align}
    \int_{\RR^N}\left|\mathcal{E}_b\left(\eta_L v_e\right)\right|^2\md x&=\int_{\left\{x\in\RR^N:\,L\leq|x|\leq2 L\right\}}\left|\mathcal{E}_b\left(\eta_L v_e\right)\right|^2\md x\nonumber\\
    &\leq C L^{-2c_e}\int_L^{2L}t^{N-1}\md t\nonumber\\
    &=C L^{N-2c_e}\nonumber\\
    &\to 0,\qquad\text{as }L\to+\oo.
    \label{cut-off for even case 1}
\end{align}
Here we use the fact that $c_e>N/2$. Moreover, when $j>1$, we need to consider the term $w_L$. Since $w_L$ is a linear combination of products of partial derivatives of $\eta_L$ and $\Phi_{b,m}$, from \eqref{estimate of eta} and Lemma \ref{estimate of Phi} we obtain the following estimate:
\begin{align}
    \left|w_L\right|\leq CL^{-N}\ln^{\sigma}(L),\qquad\left|\nabla w_L\right|\leq CL^{-N-1}\ln^{\sigma}(L),\qquad\left|\Delta w_L\right|\leq CL^{-N-2}\ln^{\sigma}(L),
    \label{estimate of wL}
\end{align}
for sufficiently large $L>0$, where $C>0$ and $\sigma\in\NN_+\cup\{0\}$ are constants independent of $L$. Therefore, we have
\begin{align}
        \int_{\RR^N}\left|\mathcal{E}_b\left(w_L\right)\right|^2\md x&=\int_{\left\{x\in\RR^N:\,L\leq|x|\leq2 L\right\}}\left|\mathcal{E}_b\left(w_L\right)\right|^2\md x\nonumber\\
    &\leq C L^{-2N}\ln^{2\sigma}(L)\int_L^{2L}t^{N-1}\md t\nonumber\\
    &=C L^{-N}\ln^{2\sigma}(L)\nonumber\\
    &\to 0,\qquad\text{as }L\to+\oo.
    \label{cut-off for even case 2}
\end{align}
Combining \eqref{cut-off for even case 1} and \eqref{cut-off for even case 2}, we conclude that
\begin{equation*}
    \lim_{L\to+\oo}\int_{\RR^N}\lvert\mathcal{E}_b\left( v_L\right)\rvert^2\md x=0.
\end{equation*}
\\
\textbf{Case 2:} $m$ is odd. In this case, we have $\Delta^{j-1}\Phi_{b,m}(x)=v_o(x,\,b)$.  Then we deduce that
\begin{align*}
    \mathcal{O}_b\left(\eta_Lv_o\right)&=\eta_L\mathcal{O}_b\left(v_o\right)+\Delta v_o\nabla\eta_L+2\nabla\left(\nabla\eta_L\cdot\nabla v_o\right)+\nabla\left(v_o\Delta\eta_L\right)\\
    &\qquad\qquad+b\left(2\nabla v_o\cdot\nabla\eta_L+v_o\Delta\eta_L\right)x+bc v_o\nabla\eta_L.
\end{align*}
From Proposition \ref{expression of v} and Lemma \ref{asymptotic of v}, we get $\mathcal{O}_b\left(v_o\right)=0$ and 
\begin{align*}
    \left|\mathcal{O}_b\left(\eta_Lv_o\right)\right|=O\left(|x|^{1-N-c_o}\right),\qquad\text{as }|x|\to+\oo.
\end{align*}
Notice that $N+2c_o>2$, We deduce
\begin{align}
        \int_{\RR^N}\left|\mathcal{O}_b\left(\eta_L v_o\right)\right|^2\md x&=\int_{\left\{x\in\RR^N:\,L\leq|x|\leq2 L\right\}}\left|\mathcal{O}_b\left(\eta_L v_o\right)\right|^2\md x\nonumber\\
    &\leq C L^{2-2N-2c_o}\int_L^{2L}t^{N-1}\md t\nonumber\\
    &=C L^{2-N-2c_o}\nonumber\\
    &\to 0,\qquad\text{as }L\to+\oo.
    \label{cut-off for odd case 1}
\end{align}
As for the term $w_L$, which needs to be considered when $j>1$, we have the estimate \eqref{estimate of wL} together with the following estimate
\begin{equation*}
    \left|\nabla\Delta w_L\right|\leq CL^{-N-3}\ln^{\sigma}(L),
\end{equation*}
for large $L>0$. Then we obtain that
\begin{align}
    \int_{\RR^N}\left|\mathcal{O}_b\left(w_L\right)\right|^2\md x&=\int_{\left\{x\in\RR^N:\,L\leq|x|\leq2 L\right\}}\left|\mathcal{O}_b\left(w_L\right)\right|^2\md x\nonumber\\
    &\leq C L^{-2N-2}\ln^{2\sigma}(L)\int_L^{2L}t^{N-1}\md t\nonumber\\
    &=C L^{-N-2}\ln^{2\sigma}(L)\nonumber\\
    &\to 0,\qquad\text{as }L\to+\oo.
    \label{cut-off for odd case 2}
\end{align}
From \eqref{cut-off for odd case 1} and \eqref{cut-off for odd case 2} we get that 
\begin{equation*}
    \lim_{L\to+\oo}\int_{\RR^N}\lvert\mathcal{O}_b\left( v_L\right)\rvert^2\md x=0.
\end{equation*}
\end{proof}
We now establish that the constant $\mu_m(N)$ in \eqref{eq m order HUP} is optimal. For $u
\in C^\oo_0\left(\RR^N\right)\backslash\{0\}$, define the quotient $\mathcal{Q}_m$ as 
\begin{equation}
\label{def of Q}
    \mathcal{Q}_m(u)\coloneqq\dfrac{\sqrt{\int_{\RR^N}\lvert P_m u\rvert^2\md x\int_{\RR^N}\lvert x\rvert^2\lvert \nabla^{m-1} u\rvert^2\md x}}{\int_{\RR^N}\lvert P_{m-1}u\rvert^2\md x}.
\end{equation}
\vspace{0.5cm}
\begin{proposition}
\label{sharpness}
Let $N\geq2$, $m\geq2$ and $b>0$. For $u_L\in C^\oo_0\left(\RR^N\right)$ defined in Lemma \ref{cut off of Phi} and $\mathcal{Q}_m$ defined in \eqref{def of Q}, there holds
\begin{equation*}
    \lim_{L\to+\oo}\mathcal{Q}_m(u_L)=\sqrt{\mu_m(N)}.
\end{equation*}
\end{proposition}
\begin{proof}
We denote $j=\lfloor m/2\rfloor$ and $v_L=\Delta^{j-1} u_L$ for $L>0$, where $v_L$ is given by \eqref{expression of vL}.
\\
\textbf{Case 1:} $m$ is even. Setting $c=c_e$
 in \eqref{expansion of E} and noting that  $v_L$ is radially symmetric, by Lemma \ref{integral identity 3}, we get
\begin{align*}
    &\int_{\RR^N}\lvert\mathcal{E}_{b}\left(v_L\right)\rvert^2\md x\\
   =&\int_{\RR^N}\lvert\Delta v_L\rvert^2\md x+b^2\int_{\RR^N}\lvert x\rvert^2\lvert\nabla v_L\rvert^2\md x+b\left(N-2c_e-2\right)\int_{\RR^N}\lvert \nabla v_L\rvert^2\md x\\
   &\hspace{2cm}+b^2\left(c_e^2-c_eN\right)\int_{\RR^N}v_L^2\md x\nonumber\\
    =&\int_{\RR^N}|P_m u_L|^2\md x+b^2\int_{\RR^N}|x|^2\left|P_{m-1}u_L\right|^2\md x+b\left(N-2c_e-2\right)\int_{\RR^N}\left|P_{m-1}u_L\right|^2\md x\nonumber\\
    &\hspace{2cm}+b^2\left(c_e^2-c_eN\right)\int_{\RR^N}\left|P_{m-2} u_L\right|^2\md x\nonumber\\
    =&\int_{\RR^N}|P_m u_L|^2\md x+b^2\int_{\RR^N}|x|^2\left|\nabla^{m-1}u_L\right|^2\md x+b\left(N-2c_e-2\right)\int_{\RR^N}\left|P_{m-1}u_L\right|^2\md x\nonumber\\
    &\hspace{2cm}+b^2\left(c_e^2-c_eN-(m-2)(N-1)\right)\int_{\RR^N}\left|P_{m-2} u_L\right|^2\md x\\
    =&\int_{\RR^N}|P_m u_L|^2\md x+b^2\int_{\RR^N}|x|^2\left|\nabla^{m-1}u_L\right|^2\md x+b\left(N-2c_e-2\right)\int_{\RR^N}\left|P_{m-1}u_L\right|^2\md x.
\end{align*}
\\
It follows that
\begin{align}
    &2b\sqrt{\int_{\RR^N}\left\lvert P_m u_L\right\rvert^2\md x\int_{\RR^N}\lvert x\rvert^2\left\lvert \nabla^{m-1} u_L\right\rvert^2\md x}\nonumber\\
    \leq& {\int_{\RR^N}\left\lvert P_m u_L\right\rvert^2\md x+b^2\int_{\RR^N}\lvert x\rvert^2\left\lvert \nabla^{m-1} u_L\right\rvert^2\md x}\nonumber\\
    =&b\left(2c_e+2-N\right)\int_{\RR^N}\left\lvert P_{m-1} u_L\right\rvert^2\md x+\int_{\RR^N}\left\lvert\mathcal{E}_b\left(v_L\right)\right\rvert^2\md x.
    \label{proof of optimal even}
\end{align}
Next we verify that 
\begin{align*}
\lim_{L\to+\oo}\int_{\RR^N}\left|P_{m-1}u_L\right|^2\md x=\int_{\RR^N}\left|P_{m-1}\Phi_{b,m}\right|^2\md x>0.
\end{align*}
\\
Here $P_{m-1}u_L=\nabla\Delta^{j-1}u_L=\nabla v_L$. By \eqref{expression of vL},  for any $j\geq1$, we have
\begin{align*}
    P_{m-1}u_L=\eta_LP_{m-1}\Phi_{b,m}+q_L,
\end{align*}
\\
where $q_L$ is a linear combination of products in the form $D^{\beta_1}\eta_L\cdot D^{\beta_2}\Phi_{b,m}$ for $|\beta_1|+|\beta_2|=2j-1$ and $|\beta_1|\geq1$. By \eqref{estimate of eta} and Lemma \ref{estimate of Phi}, we obtain $q_L=O\left(|x|^{-N-1}\ln^\sigma(|x|)\right)$ as $|x|\to+\oo$. Thus, for sufficiently large $L>0$, we have
\begin{align*}
    \int_{\RR^N}\left\lvert q_L\right\rvert^2\md x=\int_{\{x\in\RR^N:\,L\leq|x|\leq 2L\}}\left\lvert q_L\right\rvert^2\md x
    &\leq CL^{-2N-2}\ln^{2\sigma}(L)\int_L^{2L}t^{N-1}  \md t\\
    &\to0, \text{\qquad as }L\to+\oo.
\end{align*}
\\
Moreover, for $P_{m-1}\Phi_{b,m}=\nabla v_e$, there holds:
\begin{align*}
    \left\lvert \int_{\RR^N}\left\lvert \eta_LP_{m-1} \Phi_{b,m}\right\rvert^2\md x-\int_{\RR^N}\left\lvert P_{m-1}\Phi_{b,m}\right\rvert^2\md x\right\rvert&=\int_{\{x\in\RR^N:\,\lvert x\rvert\geq L \}}\left\lvert P_{m-1}\Phi_{b,m}\rvert^2\lvert\eta_L^2-1\right\rvert^2\md x\\
    &\leq C\int_L^{+\oo}r^{-2\left(1+c_e\right)}\cdot r^{N-1}\md r\\
    &\to0, \text{\qquad as }L\to+\oo.
\end{align*}
It follows that
\begin{align*}
   \lim_{L\to+\oo} \left\lvert \int_{\RR^N}\left\lvert P_{m-1}u_L\right\rvert^2\md x-\int_{\RR^N}\left\lvert P_{m-1}\Phi_{b,m}\right\rvert^2\md x\right\rvert=0.
\end{align*}
Therefore, together with Theorem  \ref{thm m order HUP} and \eqref{proof of optimal even}, we have
\begin{align*}
   \dfrac{2c_e-N+2}{2}\leq \mathcal{Q}_m\left(u_L\right)&\leq \dfrac{2c_e-N+2}{2}+\dfrac{\int_{\RR^N}\left\lvert\mathcal{E}_b\left(v_L\right)\right\rvert^2\md x}{2b \int_{\RR^N}\left\lvert P_{m-1} u_L\right\rvert^2\md x}.
\end{align*}
Now, letting $L\to+\oo$, by Lemma \ref{cut off of Phi}, there holds:
\begin{align*}
    \lim_{L\to+\oo}\mathcal{Q}_m\left(u_L\right)=\frac{2c_e-N+2}{2}=\sqrt{\mu_m(N)}.
\end{align*}
\\
\textbf{Case 2:} $m$ is odd. Taking $c=c_o$ in \eqref{linearized odd}, we obtain
\begin{align*}
    \int_{\RR^N}\left|\mathcal{O}_b\left(v_L\right)\right|^2\md x&=\int_{\RR^N}\left\lvert P_m u_L\right\rvert^2\md x+b^2\int_{\RR^N}\left\lvert \nabla^{m-1}u_L\right\rvert^2\lvert x\rvert^2\md x\nonumber\\
    &\hspace{3.2cm}-b\left(N+2c_o\right)\int_{\RR^N}\left\lvert P_{m-1}u_L\right\rvert^2\md x
\end{align*}
Consequently, we have
\begin{align}
    &2b\sqrt{\int_{\RR^N}\left\lvert P_m u_L\right\rvert^2\md x\int_{\RR^N}\lvert x\rvert^2\left\lvert \nabla^{m-1} u_L\right\rvert^2\md x}\nonumber\\
    \leq& {\int_{\RR^N}\left\lvert P_m u_L\right\rvert^2\md x+b^2\int_{\RR^N}\lvert x\rvert^2\left\lvert \nabla^{m-1} u_L\right\rvert^2\md x}\nonumber\\
    =&b\left(2c_o+N\right)\int_{\RR^N}\left\lvert P_{m-1} u_L\right\rvert^2\md x+\int_{\RR^N}\left\lvert\mathcal{O}_b\left(v_L\right)\right\rvert^2\md x.
    \label{proof of optimal odd}
\end{align}
Next we prove that $\int_{\RR^N}\left|P_{m-1}u_L\right|^2\md x$ converges to $\int_{\RR^N}\left|P_{m-1}\Phi_{b,m}\right|^2\md x$ as $L\to+\oo$. By \eqref{expression of vL}, there holds
\begin{equation*}
    P_{m-1}u_L=\Delta v_L=\eta_L P_{m-1}\Phi_{b,m}+\tilde{q}_L,
\end{equation*}
where $\tilde{q}_L$ is a linear combination of products in the form $D^{\beta_1}\eta_L\cdot D^{\beta_2}\Phi_{b,m}$ for $|\beta_1|+|\beta_2|=2j$ and $|\beta_1|\geq1$. Similarly, from  \eqref{estimate of eta} and Lemma \ref{estimate of Phi}, for all $j\geq1$,  we have 
\begin{align*}
\tilde{q}_L=O\left(|x|^{-N-1}\ln^\sigma(|x|)\right).
\end{align*}
This implies that
\begin{align*}
      \int_{\RR^N}\left\lvert \tilde{q}_L\right\rvert^2\md x=\int_{\{x\in\RR^N:\,L\leq|x|\leq 2L\}}\left\lvert \tilde{q}_L\right\rvert^2\md x
    &\leq CL^{-2N-2}\ln^{2\sigma}(L)\int_L^{2L}t^{N-1}  \md t\\
    &\to0, \text{\qquad as }L\to+\oo.
\end{align*}
In addition, We also have the following estimate for $P_{m-1}\Phi_{b,m}=\Delta v_o$:
\begin{align*}
    \left\lvert \int_{\RR^N}\left\lvert \eta_LP_{m-1} \Phi_{b,m}\right\rvert^2\md x-\int_{\RR^N}\left\lvert P_{m-1}\Phi_{b,m}\right\rvert^2\md x\right\rvert&=\int_{\{x\in\RR^N:\,\lvert x\rvert\geq L \}}\left\lvert P_{m-1}\Phi_{b,m}\rvert^2\lvert\eta_L^2-1\right\rvert^2\md x\\
    &\leq C\int_L^{+\oo}r^{-2\left(N+c_o\right)}\cdot r^{N-1}\md r\\
    &\to0, \text{\qquad as }L\to+\oo.
\end{align*}
It follows that
\begin{align*}
   \lim_{L\to+\oo} \left\lvert \int_{\RR^N}\left\lvert P_{m-1}u_L\right\rvert^2\md x-\int_{\RR^N}\left\lvert P_{m-1}\Phi_{b,m}\right\rvert^2\md x\right\rvert=0.
\end{align*}
Hence, by Theorem  \ref{thm m order HUP} and \eqref{proof of optimal odd}, we have
\begin{align*}
   \dfrac{2c_o+N}{2}\leq \mathcal{Q}_m\left(u_L\right)&\leq \dfrac{2c_o+N}{2}+\dfrac{\int_{\RR^N}\left\lvert\mathcal{O}_b\left(v_L\right)\right\rvert^2\md x}{2b \int_{\RR^N}\left\lvert P_{m-1} u_L\right\rvert^2\md x}.
\end{align*}
Finally, by Lemma \ref{cut off of Phi}, there holds:
\begin{align*}
    \lim_{L\to+\oo}\mathcal{Q}_m\left(u_L\right)=\frac{2c_o+N}{2}=\sqrt{\mu_m(N)}.
\end{align*}
\end{proof}

Similarly, when $N=1$, we have the following result.
\begin{proposition}
\label{sharpness when N=1}
Let $N=1$, $m\geq2$ and $b>0$. Denote $u_L=\eta_L\Phi_{b,m}\in C^\oo_0\left(\RR\right)$ for $L>0$, where  $\Phi_{b,m}$ is given by \eqref{def of Phi when N=1} and $\eta_L$ is given by \eqref{def of eta L}. For $\mathcal{Q}_m$ defined in \eqref{def of Q}, there holds
\begin{equation*}
    \lim_{L\to+\oo}\mathcal{Q}_m(u_L)=\mathcal{Q}_m\left(\Phi_{b,m}\right)=\sqrt{\mu_m(1)}=\frac{3}{2}.
\end{equation*}
\end{proposition}
\begin{proof}
In fact, for $k\in\NN_+$, Lemma \ref{estimate of Lk} implies that
\begin{align*}
    \left(u_L-\Phi_{b,m}\right)^{(k)}(x)=\sum_{l=0}^{k}C(k,\,l)\left(\eta_L-1\right)^{(l)}(x)\cdot\Phi_{b,m}^{(k-l)}(x)=O\left(|x|^{m-3-k}\right).
\end{align*}
Then, as $L\to+\oo$, we have
\begin{align*}
    \int_{\RR}\left|\left(u_L-\Phi_{b,m}\right)^{(m)}(x)\right|^2\md x&\leq C\int_{\{x\in\RR:\,|x|\geq L\}}|x|^{-6}\md x\to0,\\
    \int_{\RR}x^2\left|\left(u_L-\Phi_{b,m}\right)^{(m-1)}(x)\right|^2\md x&\leq C\int_{\{x\in\RR:\,|x|\geq L\}}|x|^{-2}\md x\to0,\\
    \int_{\RR}\left|\left(u_L-\Phi_{b,m}\right)^{(m-1)}(x)\right|^2\md x&\leq C\int_{\{x\in\RR:\,|x|\geq L\}}|x|^{-4}\md x\to0.
\end{align*}
Hence, using $\Phi_{b,m}^{(m-2)}(x)=\mathrm{e}^{-bx^2}$, we deduce that
\begin{align*}
    \lim_{L\to+\oo}\int_{\RR}\left|u_L^{(m)}(x)\right|^2\md x&=\int_{\RR}\left|\Phi_{b,m}^{(m)}(x)\right|^2\md x=3b^2\sqrt{\frac{\pi}{2b}},\\
    \lim_{L\to+\oo}\int_{\RR}x^2\left|u^{(m-1)}_L(x)\right|^2\md x&=\int_{\RR}x^2\left|\Phi_{b,m}^{(m-1)}(x)\right|^2\md x=\frac{3}{4}\sqrt{\frac{\pi}{2b}},\\
    \lim_{L\to+\oo}\int_{\RR}\left|u_L^{(m-1)}(x)\right|^2\md x&=\int_{\RR}\left|\Phi_{b,m}^{(m-1)}(x)\right|^2\md x=b\sqrt{\frac{\pi}{2b}}.
\end{align*}
It follows that
\begin{align*}
    \lim_{L\to+\oo}\mathcal{Q}_m(u_L)=\mathcal{Q}_m\left(\Phi_{b,m}\right)=\frac{3}{2}.
\end{align*}
\end{proof}

Now we can give the proof of Theorem \ref{extremal}.
\begin{proof}[Proof of Theorem \ref{extremal}]\hfill

The case $N=1$ follows directly from Theorem \ref{thm m order HUP} and Proposition \ref{sharpness when N=1}. Thus, we restrict ourselves to the case $N\geq2$ in what follows. In fact, by Theorem \ref{thm m order HUP} and Proposition \ref{sharpness}, we deduce that $\mu_m(N)$ in \eqref{eq m order HUP} is optimal. Therefore, we only need to prove that $\Phi_{b,m}$ given by \eqref{def of extremal} is an extremal function, that is, $\mathcal{Q}_m\left(\Phi_{b,m}\right)=\sqrt{\mu_m(N)}$. Notice that  \eqref{relationship between Phi and v} implies
\begin{align*}
    P_{m-2}\Phi_{b,m}(x)=\begin{cases}
    \mathrm{e}^{-\frac{br^2}{2}}{_1F_1}\left(\dfrac{N-\sqrt{N^2+4(m-2)(N-1)}}{4},\,\dfrac{N}{2},\,\dfrac{br^2}{2}\right),\,&\text{if $m$ is even,}\\
     x\mathrm{e}^{-\frac{br^2}{2}}{_1F_1}\left(\dfrac{2+N-\sqrt{N^2+4(m-2)(N-1)}}{4},\,\dfrac{N+2}{2},\,\dfrac{br^2}{2}\right),\,&\text{if $m$ is odd,}
    \end{cases}
\end{align*}
\text{for }$r=|x|$. Besides, by Lemma \ref{estimate of Phi}, as $|x|\to+\oo$, we deduce that
\begin{align*}
    \left|\nabla^{m-1}\Phi_{b,m} \right|=O\left(|x|^{-N-1}\ln^{\sigma}|x|\right),\qquad P_{m}\Phi_{b,m}=O\left(|x|^{-N-2}\ln^\sigma|x|\right).
\end{align*}
Then we find $\left\|\Phi_{b,m}\right\|_{\mathcal{H}^m}<+\oo$. Next, for $\|\cdot\|_{\mathcal{H}^m}$ given by \eqref{def of seminorm} and $u_L=\eta_L\Phi_{b,m}$ defined in Lemma \ref{cut off of Phi},  we will prove that 
\begin{align*}
    \lim_{L\to+\oo}\|u_L-\Phi_{b,m}\|_{\mathcal{H}^m}=0.
\end{align*}
Note that
\begin{align*}
    P_{m}\left(u_L-\Phi_{b,m}\right)=\left(\eta_L-1\right)P_m\Phi_{b,m}+\rho_L,
\end{align*}
where $\rho_L$ is a linear combination of products in the form $D^{\beta_1}\eta_L\cdot D^{\beta_2}\Phi_{b,m}$ for $|\beta_1|+|\beta_2|=m$ and $|\beta_1|\geq1$. From  \eqref{estimate of eta} and Lemma \ref{estimate of Phi}, we have
\begin{align*}
    \rho_L=O\left(|x|^{-N-2}\ln^{\sigma}|x|\right)\qquad\text{as }|x|\to+\oo.
\end{align*}
Hence, for large $L>0$, we get
\begin{align}
    \int_{\RR^N}\left|P_{m}\left(u_L-\Phi_{b,m}\right)\right|^2\md x&\leq 2\int_{\RR^N}\left|\eta_L-1\right|^2\left|P_m\Phi_{b,m}\right|^2\md x+2\int_{\RR^N}\left|\rho_L\right|^2\md x\nonumber\\
    &=2\int_{\{x\in\RR^N:\,|x|\geq L\}}\left|P_m\Phi_{b,m}\right|^2\md x+2\int_{\{x\in\RR^N:\,L\leq|x|\leq 2L\}}\left|\rho_L\right|^2\md x\nonumber\\
    &\leq C\int_L^{+\oo}t^{-2N-4}\ln^{2\sigma}(t)\cdot t^{N-1}\md t\nonumber\\
    &\to0\qquad\text{as }L\to+\oo.
    \label{converge in norm 1}
\end{align}
Additionally, we have
\begin{align*}
    \nabla^{m-1}\left(u_L-\Phi_{b,m}\right)&=\nabla^{m-1}\left(\left(\eta_L-1\right)\Phi_{b,m}\right)
    =\left(\eta_L-1\right)\nabla^{m-1}\Phi_{b,m}+T,
\end{align*}
where $T$ is an $(m-1)$-th rank tensor.  Omitting the constant factors, each component of $T$ is of the form
\begin{equation*}
     D^{\beta_1}\eta_L\cdot D^{\beta_2}\Phi_{b,m}\qquad\text{for }\left|\beta_1\right|+\left|\beta_2\right|=m-1\text{ and }\left|\beta_1\right|\geq1.
\end{equation*}
Then we have $D^{\beta_1}\eta_L\cdot D^{\beta_2}\Phi_{b,m}=O\left(|x|^{-N-1}\ln^\sigma|x|\right)$ as $|x|\to+\oo.$  Therefore, for large $L>0$, we have the following result:
\begin{align}
    \label{converge in norm 2}
    \int_{\RR^N}|x|^2\left|\nabla^{m-1}\left(u_L-\Phi_{b,m}\right)\right|^2\md x&\leq 2\int_{\RR^N}|x|^2\left|\eta_L-1\right|^2\left|\nabla^{m-1}\Phi_{b,m}\right|^2\md x+2\int_{\RR^N}|x|^2\left|T\right|^2\md x\nonumber\\
    &\leq C\int_L^{+\oo}t^{-2N}\ln^{2\sigma}(t)\cdot t^{N-1}\md t\nonumber\\
    &\to0\qquad\text{as }L\to+\oo.
\end{align}
Combining \eqref{converge in norm 1} and \eqref{converge in norm 2}, we have
\begin{equation*}
       \lim_{L\to+\oo}\|u_L-\Phi_{b,m}\|_{\mathcal{H}^m}=0.
\end{equation*}
Finally, since we proved in Proposition  \ref{sharpness} that
\begin{align*}
\lim_{L\to+\oo} \left\lvert \int_{\RR^N}\left\lvert P_{m-1}u_L\right\rvert^2\md x-\int_{\RR^N}\left\lvert P_{m-1}\Phi_{b,m}\right\rvert^2\md x\right\rvert=0,
\end{align*}
we obtain 
\begin{equation*}
    \mathcal{Q}_m\left(\Phi_{b,m}\right)=\lim_{L\to+\oo}\mathcal{Q}_m\left(u_L\right)=\sqrt{\mu_m(N)}.
\end{equation*}
Then, we conclude that $\Phi_{b,m}\in\mathcal{H}^m$ is an  extremal function  for the $m$-th order uncertainty principle.
\end{proof}
\begin{proof}[Proof of Corollary \ref{cor m order HUP}]\hfill

Let $U$ be a smooth closed 1-form and compactly supported in $\RR^N$. Since $\RR^N$ is contractible, by Poincar\'e lemma, $U$ is also exact (which also holds for $N=1$). Then we have $\vec{U}=\nabla u$ for some $u\in C^\oo_0\left(\RR^N\right)$. By Lemme \ref{integral identity 1}, we have
\begin{align*}
    \int_{\RR^N}\left|P_{k}\vec{U}\right|^2\md x=\sum_{i=1}^N\int_{\RR^N}\left|P_{k}\frac{\pt u}{\pt x_i}\right|^2\md x&=\sum_{i=1}^N\int_{\RR^N}\left|\nabla^{k}\frac{\pt u}{\pt x_i}\right|^2\md x\\
    &=\int_{\RR^N}\left|\nabla^{k+1}u\right|^2\md x\\
    &=\int_{\RR^N}\left|P_{k+1}{u}\right|^2\md x,
\end{align*}
for $k\in \NN_+$. Therefore, by Theorem \ref{thm m order HUP}, there holds
\begin{align*}
    \int_{\RR^N}|P_{m-1}\vec{U}|^2\md x\int_{\RR^N}|x|^2|\nabla^{m-2}\vec{U}|^2\md x&=\int_{\RR^N}\left|P_{m}u\right|^2\md x\int_{\RR^N}|x|^2|\nabla^{m-1}u|^2\md x\\
    &\geq \mu_m(N)\left(\int_{\RR^N}\left|P_{m-1}u\right|^2\md x\right)^2\\
    &=\mu_{m}(N)\left(\int_{\RR^N}|P_{m-2}\vec{U}|^2\md x\right)^2
\end{align*}
Furthermore, by Theorem \ref{extremal}, the constant $\mu_m(N)$ in \eqref{eq m order HUP} is attained by some $\Phi_{b,m}$. Applying Proposition \ref{sharpness} and Proposition \ref{sharpness when N=1} to  $u_L=\eta_L\Phi_{b,m}$ and $\vec{U}_L=\nabla u_L$ for $L>0$, we obtain
\begin{align*}
    \frac{\int_{\RR^N}|P_{m-1}\vec{U}_L|^2\md x\int_{\RR^N}|x|^2|\nabla^{m-2}\vec{U}_L|^2\md x}{\left(\int_{\RR^N}|P_{m-2}\vec{U}_L|^2\md x\right)^2}&=\frac{\int_{\RR^N}|P_{m}{u_L}|^2\md x\int_{\RR^N}|x|^2|\nabla^{m-1}{u_L}|^2\md x}{\left(\int_{\RR^N}|P_{m-1}{u_L}|^2\md x\right)^2}\\
    &\to \mu_m(N)\qquad\text{as }L\to+\oo.
\end{align*}
Thus, the constant $\mu_m(N)$ in \eqref{equation of cor m HUP} is optimal.
\end{proof}
\begin{proof}[Proof of Corollary \ref{cor Mazya}]
\hfill

Set $N=2$. If $\vec{U}\in C^\oo_0\left(\RR^2;\,\RR^2\right)$ and $\mathrm{div}\,\vec{U}=0$, then 
\begin{equation*}
    \vec{U}=\left(\frac{\pt u}{\pt x_2},\,-\frac{\pt u}{\pt x_1}\right)\qquad\text{for some }u\in C^\oo_0\left(\RR^2\right).
\end{equation*}
By Lemme \ref{integral identity 1}, a direct computation shows that
\begin{align*}
    \int_{\RR^2}\left|P_{k}\vec{U}\right|^2\md x=\sum_{i=1}^2\int_{\RR^2}\left|P_{k}\frac{\pt u}{\pt x_i}\right|^2\md x&=\sum_{i=1}^2\int_{\RR^2}\left|\nabla^{k}\frac{\pt u}{\pt x_i}\right|^2\md x\\
    &=\int_{\RR^2}\left|\nabla^{k+1}u\right|^2\md x\\
    &=\int_{\RR^2}\left|P_{k+1}{u}\right|^2\md x,
\end{align*}
for $k\in \NN_+$. Hence, \eqref{equation of cor Mazya} follows from Theorem \ref{thm m order HUP}.

Next, to prove that the constant $\mu_m(2)$ in \eqref{equation of cor Mazya} is optimal, we need to use the extremal function $\Phi_{b,m}$ given in Theorem \ref{extremal} and the cut-off function $\eta_L$ given by \eqref{def of eta L}. Define
\begin{equation*}
    \vec{U}_L=\left(\frac{\pt u_L}{\pt x_2},\,-\frac{\pt u_L}{\pt x_1}\right)\qquad\text{for }u_L=\eta_L\Phi_{b,m}\in C^\oo_0\left(\RR^2\right).
\end{equation*}
It follows that $\mathrm{div}\,\vec{U}_L=0$. Then,  Proposition \ref{sharpness} implies that
\begin{align*}
    \frac{\int_{\RR^2}|P_{m-1}\vec{U}_L|^2\md x\int_{\RR^2}|x|^2|\nabla^{m-2}\vec{U}_L|^2\md x}{\left(\int_{\RR^2}|P_{m-2}\vec{U}_L|^2\md x\right)^2}&=\frac{\int_{\RR^2}|P_{m}{u_L}|^2\md x\int_{\RR^2}|x|^2|\nabla^{m-1}{u_L}|^2\md x}{\left(\int_{\RR^2}|P_{m-1}{u_L}|^2\md x\right)^2}\\
    &\to \mu_m(2)\qquad\text{as }L\to+\oo.
\end{align*}
\end{proof}
\vspace{0.4cm}
\section{Proofs of Theorem \ref{thm m order HUP without tensor} and Corollary \ref{cor m order HUP without tensor}}\label{sec 6}
In this section, we will establish a version of the Heisenberg uncertainty principle without higher-order tensors  and proof the existence of the extremal functions. In fact, taking $m=2$ in Theorem \ref{thm m order HUP}, we have
\begin{equation}
    \label{equation of 2nd order HUP}
    \int_{\RR^N}\left|\Delta u\right|^2\md x\int_{\RR^N}|x|^2|\nabla u|^2\md x\geq\frac{(N+2)^2}{4}\left(\int_{\RR^N}|\nabla u|^2\md x\right)^2\qquad\text{for $u\in C^\oo_0\left(\RR^N\right)$}.
\end{equation}
Recall the classical first-order Heisenberg uncertainty principle:
\begin{equation}
\label{equation of 1st order HUP}
    \int_{\RR^N}|\nabla u|^2\md x\int_{\RR^N}|x|^2|u|^2\md x\geq\frac{N^2}{4}\left(\int_{\RR^N}|u|^2\md x\right)^2\qquad\text{for $u\in C^\oo_0\left(\RR^N\right)$}.
\end{equation}

\begin{proof}[Proof of Theorem \ref{thm m order HUP without tensor}]\hfil

For $N,\,m\geq2$ and $u\in C^\oo_0\left(\RR^N\right)$, we substitute $P_{m-2}u$ into \eqref{equation of 2nd order HUP} when $m$ is even, and substitute $P_{m-1}u$ into \eqref{equation of 1st order HUP} when $m$ is odd. Together with a density argument, we obtain 
\begin{align}
     \int_{\RR^N}\left|P_mu\right|^2\md x\int_{\RR^N}|x|^2\left|P_{m-1}u\right|^2\md x\geq \mu^\star_m(N)\left(\int_{\RR^N}\left|P_{m-1}u\right|^2\md x\right)^2\quad\text{for $u\in\mathcal{H}^m_1$},
     \label{equation m order HUP without tensor 2}
\end{align}
where $\mu^\star_m(N)=\dfrac{(N+2)^2}{4}$ if $m$ is even and $\mu_m^\star(N)=\dfrac{N^2}{4}$ if $m$ is odd. Next, we prove the optimality of the constant $\mu_m^\star(N)$. Note that $u(x)=\mathrm{e}^{-b|x|^2}\,(b>0)$ are  extremal functions for both \eqref{equation of 2nd order HUP} and \eqref{equation of 1st order HUP}. We will construct extremal functions for \eqref{equation m order HUP without tensor 2} using the operator $\mathcal{I}$ defined by \eqref{def of I}. In particular, when $m=2$, \eqref{equation m order HUP without tensor 2} is equivalent to \eqref{eq m order HUP}. Hence, we only need to consider $m\geq3$ in the following. In view of the asymptotic estimates established in Section \ref{sec 4}, we shall analyze the following two cases separately.\\
\textbf{Case 1:} $m$ is even and $N\geq2$, or $m$ is odd and $N\geq3$.

Notice that $f(t)=\mathrm{e}^{-bt^2}\in \mathcal{W}(a,\,0)$ for any $a>0$ and $b>0$. Therefore, by repeating the induction argument in Proposition \ref{expression of Ik}, for $k\in\NN_+$,  we obtain:
\begin{align*}
    \mathcal{I}^k\left[\mathrm{e}^{-bt^2}\right](r)&=\sum_{i=1}^{n_k}f_i(r)+Q_{2k-2}(r) \qquad\text{for large }r>0,
\end{align*}
where $f_i\in \mathcal{W}\left(N-2k,\,\sigma_i\right)$  with $\sigma_i\in\NN_+\cup\{0\}$ and  $Q_{2k-2}$ is a polynomial of degree at most $(2k-2)$ containing only even powers. For $j=\lfloor m/2\rfloor$, $r=|x|$ and $b>0$, we define
\begin{align}
      \Phi^\star_{b,m}(x)=\Phi^\star_{b,m}(r)=\begin{cases}
    \mathcal{I}^{j-1}\left[\mathrm{e}^{-bt^2}\right](r)-Q_{2j-4}(r),&\text{ if } m=2j,\\
    \mathcal{I}^{j}\left[\mathrm{e}^{-bt^2}\right](r)-{Q}_{2j-2}(r),&\text{ if } m=2j+1.
    \end{cases}
    \label{def of Phi star}
\end{align}
It follows that $P_{m-2}\Phi_{b,m}^\star=\mathrm{e}^{-br^2}$ if $m=2j$ and $P_{m-1}\Phi_{b,m}^\star=\mathrm{e}^{-br^2}$ if $m=2j+1$. Then we have
\begin{align*}
    \frac{\int_{\RR^N}\left|P_m\Phi_{b,m}^\star\right|^2\md x\int_{\RR^N}|x|^2\left|P_{m-1}\Phi_{b,m}^\star\right|^2\md x}{\left(\int_{\RR^N}\left|P_{m-1}\Phi_{b,m}^\star\right|^2\md x\right)^2}=\mu^\star_m(N).
\end{align*}

Next we verify that $\Phi_{b,m}^\star$ can be approximated by smooth, compactly supported functions with respect to the norm $\|\cdot\|_{\mathcal{H}^m_1}$ given in \eqref{def of seminorm 1}. By \eqref{def of Phi star}, there exists some $\sigma\in\NN_+$ depending on $m$ such that
\begin{align*}
    \Phi_{b,m}^\star(r)&\in \mathcal{W}(N-m+2,\,\sigma),\qquad\text{if }m=2j,\\
    \Phi_{b,m}^\star(r)&\in\mathcal{W}(N-m+1,\,\sigma),\qquad\text{if }m=2j+1.
\end{align*}
For any multi-index $\beta$, by \eqref{derivative of radial function}, as $|x|\to+\oo$, there holds
\begin{equation}
    D^\beta\Phi^\star_{b,m}(x)=\begin{cases}O\left(|x|^{m-2-N-|\beta|}\ln^{\sigma}|x|\right), \qquad&\text{if }m=2j,\\
    O\left(|x|^{m-1-N-|\beta|}\ln^{\sigma}|x|\right), &\text{if }m=2j+1.
    \end{cases}
    \label{estimate of Phi star}
\end{equation}
Define $u^\star_L=\eta_L\Phi^\star_{b,m}\in C^\oo_{0,\text{rad}}\left(\RR^N\right)$ for $\eta_L$ defined by \eqref{def of eta L} and $L>0$. It suffices to prove the following facts:
\begin{align}
\label{approximate Phi star}
    \begin{cases}
    \lim\limits_{L\to+\oo}\int_{\RR^N}\left|P_{m}\left(u^\star_L-\Phi_{b,m}^\star\right)\right|^2\md x=0,\\
     \lim\limits_{L\to+\oo}\int_{\RR^N}|x|^2\left|P_{m-1}\left(u^\star_L-\Phi_{b,m}^\star\right)\right|^2\md x=0,\\
      \lim\limits_{L\to+\oo}\int_{\RR^N}\left|P_{m-1}\left(u_L^\star-\Phi_{b,m}^\star\right)\right|^2\md x=0.
    \end{cases}
\end{align}
Notice that
\begin{align*}
    P_{m}\left(u^\star_L-\Phi^\star_{b,m}\right)=\left(\eta_L-1\right)P_m\Phi^\star_{b,m}+\rho^\star_L,
\end{align*}
where $\rho^\star_L$ is a linear combination of products in the form $D^{\beta_1}\eta_L\cdot D^{\beta_2}\Phi^\star_{b,m}$ for $|\beta_1|+|\beta_2|=m$ and $|\beta_1|\geq1$. From  \eqref{estimate of eta} and \eqref{estimate of Phi star}, as $L\to+\oo$, we have
\begin{align*}
    \int_{\RR^N}\left|P_{m}\left(u^\star_L-\Phi_{b,m}^\star\right)\right|^2\md x\leq\begin{cases}
    C\int_L^{+\oo}t^{-2N-4}\ln^{2\sigma}(t)\cdot t^{N-1}\md t\to0,\qquad\text{if }m=2j,\\
    C\int_L^{+\oo}t^{-2N-2}\ln^{2\sigma}(t)\cdot t^{N-1}\md t\to0,\qquad\text{if }m=2j+1.
    \end{cases}
\end{align*}
Similarly, we have
\begin{align*}
    \int_{\RR^N}\left|P_{m-1}\left(u^\star_L-\Phi_{b,m}^\star\right)\right|^2\md x\leq\begin{cases}
    C\int_L^{+\oo}t^{-2N-2}\ln^{2\sigma}(t)\cdot t^{N-1}\md t\to0,\qquad\text{if }m=2j,\\
    C\int_L^{+\oo}t^{-2N}\ln^{2\sigma}(t)\cdot t^{N-1}\md t\to0,\qquad\text{if }m=2j+1.
    \end{cases}
\end{align*}
\begin{align*}
    \int_{\RR^N}|x|^2\left|P_{m-1}\left(u^\star_L-\Phi_{b,m}^\star\right)\right|^2\md x\leq\begin{cases}
    C\int_L^{+\oo}t^{-2N}\ln^{2\sigma}(t)\cdot t^{N-1}\md t\to0,\qquad\text{if }m=2j,\\
    C\int_L^{+\oo}t^{-2N+2}\ln^{2\sigma}(t)\cdot t^{N-1}\md t\to0,\qquad\text{if }m=2j+1.
    \end{cases}
\end{align*}
Thus \eqref{approximate Phi star} holds and $\mu^\star_m(N)$ is attained by $\Phi_{b,m}^\star\in \mathcal{H}^m_1$. Note that when $m=2j+1$ and $N=2$, we cannot deduce $$\int_{\RR^N}|x|^2\left|P_{m-1}\left(u^\star_L-\Phi_{b,m}^\star\right)\right|^2\md x\to0$$ from the above estimate. Therefore, a more refined analysis is required in the next case.\\
\textbf{Case 2:} $m$ is odd and $N=2$.

Denote $j=\lfloor m/2\rfloor$. In this case, we find that the truncation described above is not sufficient to approximate $\Phi_{b,m}^\star$. To address this problem, we need to provide a more refined description of $\mathcal{I}^j\left[\mathrm{e}^{-bt^2}\right]$ when $N=2$. For $k\in\NN_+$, we will prove the following fact by induction:
\begin{align}
    \mathcal{I}^k\left[\mathrm{e}^{-bt^2}\right](r)=\tilde{Q}_{2k-2}(r)\ln(r)+Q_{2k-2}(r)+R_k(r)\qquad \text{for large $r>0$,}
    \label{expression of Ik when N=2}
\end{align}
where $Q_{2k-2}$ and $\tilde{Q}_{2k-2}$ are  polynomials of degree at most $(2k-2)$ containing only even powers and $R_k(r)\in\mathcal{W}(a,\,0)$ for any $a>0$. In fact, since  $f(t)=\mathrm{e}^{-bt^2}\in \mathcal{W}(a,\,0)$ for any $a>0$, it follows from the proof of \textbf{Case 1} in Lemma \ref{estimate of I} that
\begin{align*}
    \mathcal{I}\left[\mathrm{e}^{-bt^2}\right](r)=\tilde{C}\ln(r)+C+R_1(r)\qquad \text{for large $r>0$,}
\end{align*}
where $\mathcal{R}_1(r)\in \mathcal{W}(a,\,0)$ for any $a>0$, and the constants $C$ and $\tilde{C}$ are independent of $r$. Assuming that \eqref{expression of Ik when N=2} is true for $k$ and large $r>r_0$, we will show that it remains true for $k+1$. Notice that
\begin{align*}
    \mathcal{I}^{k+1}\left[\mathrm{e}^{-bt^2}\right](r)&=\int_{0}^{r_0}s^{-1}\int_0^s t\cdot\mathcal{I}^k\left[\mathrm{e}^{-bt^2}\right](t)\md t\md s\\
    &\qquad+\int_{r_0}^rs^{-1}\left[\int_0^{r_0}t\cdot\mathcal{I}^k\left[\mathrm{e}^{-bt^2}\right](t)\md t+\int_{r_0}^st\cdot\mathcal{I}^k\left[\mathrm{e}^{-bt^2}\right](t)\md t\right]\md s\\
    &=C+\int_{r_0}^rs^{-1}\left[C+\int_{r_0}^st\cdot\mathcal{I}^k\left[\mathrm{e}^{-bt^2}\right](t)\md t\right]\md s\\
    &=C+C'\ln(r)+\int_{r_0}^rs^{-1}\int_{r_0}^st\cdot\mathcal{I}^k\left[\mathrm{e}^{-bt^2}\right](t)\md t\md s.
\end{align*}
Substitute $\mathcal{I}^k\left[\mathrm{e}^{-bt^2}\right](r)=\tilde{Q}_{2k-2}(r)\ln(r)+Q_{2k-2}(r)+R_k(r)$ into the above integral. In particular, for $R_k\in\mathcal{W}\left(a,\,0\right)$, repeating the proof of \textbf{Case 1} in  Lemma \ref{estimate of I}, except that the lower limits of the two integrals in $\mathcal{I}$ are replaced by $r_0$, we obtain
\begin{align*}
    \int_{r_0}^rs^{-1}\int_{r_0}^s{t} R_k(t)\md t\md s =C+C'\ln(r)+\tilde{R}_{k+1}(r),
\end{align*}
where $\tilde{R}_{k+1}(r)\in \mathcal{W}(a,\,0)$ for any $a>0$. It follows that
\begin{align*}
    \mathcal{I}^{k+1}\left[\mathrm{e}^{-bt^2}\right](r)=\tilde{Q}_{2k}(r)\ln(r)+Q_{2k}(r)+R_{k+1}(r),
\end{align*}
for large $r>0$ and $\tilde{R}_{k+1}(r)\in \mathcal{W}(a,\,0)$ for any $a>0$. Hence, \eqref{expression of Ik when N=2} is true.  Now, for $r=|x|$ and $b>0$, we define:
\begin{align}
      \Phi^\star_{b,m}(x)=\Phi^\star_{b,m}(r)=\mathcal{I}^{j}\left[\mathrm{e}^{-bt^2}\right](r)-{Q}_{2j-2}(r).
    \label{def of Phi star when N=2}
\end{align}
Then $\Phi_{b,m}^\star(r)\in \mathcal{W}(2-2j,\,1)$. Next, we introduce a new cut-off function. Let $\eta^\star\in C_0^\oo\left(\RR;\,[0,1]\right)$ satisfying $\eta^\star=1$ for $|x|\leq 1/4$ and $\eta^\star=0$ for $|x|\geq 3/4$. For $L>0$, we define
\begin{equation*}
    \eta_L^\star(x)=\eta_L^\star(r)=\begin{cases}
    1,&\text{if $|x|\leq L$},\\
    \eta^\star\left(\frac{\ln|x|}{\ln L}-1\right),&\text{if $L<|x|< L^2$},\\
    0,&\text{if $|x|\geq L^2$}.
    \end{cases}
\end{equation*}
For $L<|x|<L^2$ and $k\in\NN_+$, by Fa\`a di Bruno's formula (see \cite{faadibruno1857}), we have
\begin{align}
    \frac{\md^k }{\md r^k}\eta_L^\star(r) =\frac{\left(\eta^\star\right)'\left(\frac{\ln r}{\ln L}-1\right)}{\ln L}\left(\ln r\right)^{(k)}+O\left(\frac{1}{r^k\ln^2 L}\right).
    \label{derivative of eta star}
\end{align}
Here $\left|O\left(\frac{1}{r^k\ln^2 L}\right)\right|\leq C\frac{1}{r^k\ln^2 L}$ when $L$ is sufficiently large. Define  $u^\star_L=\eta^\star_L\Phi^\star_{b,m}\in C^{\oo}_{0,\text{rad}}\left(\RR^2\right)$. As $L\to+\oo$, we have
\begin{align}
\label{approximate Phi star 1 when N=2}
    \int_{\RR^2}\left|P_{m}\left(u^\star_L-\Phi_{b,m}^\star\right)\right|^2\md x\leq \frac{C}{\ln^2 L}\int_L^{+\oo}t^{-6}\ln^{2}(t)\cdot t\md t\to0, 
\end{align}
Similarly, we have
\begin{align}
\label{approximate Phi star 2 when N=2}
    \int_{\RR^2}\left|P_{m-1}\left(u^\star_L-\Phi_{b,m}^\star\right)\right|^2\md x\leq \frac{C}{\ln^2 L}\int_L^{+\oo}t^{-4}\ln^{2}(t)\cdot t\md t\to0.
\end{align}
Finally, we need to estimate $\int_{\RR^2}|x|^2\left|P_{m-1}\left(u^\star_L-\Phi_{b,m}^\star\right)\right|^2\md x$. In fact, by \eqref{expression of Ik when N=2} and \eqref{def of Phi star when N=2}, for large $L>0$, we have
\begin{align*}
    u^\star_L(x)-\Phi_{b,m}^\star(x)=\left(\eta_L^\star(x)-1\right)\Phi^\star_{b,m}(x)=\left(\eta_L^\star(x)-1\right)\left[\tilde{Q}_{2j-2}(|x|)\ln|x|+R_j(|x|)\right].
\end{align*}
Without loss of generality, we may assume that the  coefficient of $r^{2j-2}$ in $\tilde Q_{2j-2}(r)$ is 1. Then, let  $w_1(r)=r^{2j-2}\ln r$, we have $w_2(r)=\left(\Phi_{b,m}^{\star}(r)- w_1(r)\right)\in\mathcal{W}(4-2j,\,1)$ and 
\begin{align*}
    \int_{\RR^2}|x|^2\left|P_{m-1}\left(\eta^\star_Lw_2-w_2\right)\right|^2\md x\leq \frac{C}{\ln^2 L}\int_L^{+\oo}t^{-6}\ln^{2}(t)\cdot t\md t\to0\qquad\text{as }L\to+\oo.
\end{align*}
Thus, we only need to consider $\int_{\RR^2}|x|^2\left|P_{m-1}\left(\eta^\star_L w_1-w_1\right)\right|^2\md x$. By a direct computation, we have $\Delta^j w_1=\Delta^{j}\left(r^{2j-2}\ln r\right)=0$ and
\begin{align*}
    P_{m-1}\left(\eta^\star_L w_1-w_1\right)=\Delta^j\left(\eta^\star_L w_1-w_1\right)=\sum D^{\beta_1}\eta_L^\star\cdot D^{\beta_2}w_1,
\end{align*}
where $|\beta_1|+|\beta_2|=2j$ and $|\beta_1|>0$. Recall that by \eqref{derivative of radial function}, for radially symmetric function $\eta^\star_L(r)$, we have the following
\begin{align}
\label{partial derivative of eta star}
    D^\beta \eta^\star_L(x)=\sum_{l=1}^{|\beta|}f_{\beta,l}(x){\eta^\star_L}^{(l)}(r).
\end{align}
Here $f_{\beta,l}$ satisfies
\begin{equation*}
    |f_{\beta,l}(x)|\leq C(\beta,\,l)|x|^{l-|\beta|}.
\end{equation*}
Note that $P_{m-1}\left(\eta^\star_L w_1-w_1\right)=0$ outside the ring $\Omega=\left\{x\in\RR^2:\,L\leq|x|\leq L^2\right\}$. By \eqref{derivative of eta star} and \eqref{partial derivative of eta star}, together with the fact that $\Delta^j \left(r^{2j-2}\ln^2r\right)=C/r^2$, we have
\begin{align*}
    \int_{\RR^2}|x|^2\left|P_{m-1}\left(\eta^\star_L w_1-w_1\right)\right|^2\md x&=\int_{\Omega}|x|^2\left|\sum_{\beta_1,\beta_2} \sum_{l=1}^{|\beta_1|}f_{\beta_1,l}(x){\eta^\star_L}^{(l)}(r)\cdot D^{\beta_2}w_1\right|^2\md x\\
    &\leq C \int_{\Omega}|x|^2\left|\sum_{\beta_1,\beta_2} \sum_{l=1}^{|\beta_1|}f_{\beta_1,l}(x)O\left(\frac{1}{r^l\ln^2L}\right)\cdot D^{\beta_2}w_1\right|^2\md x\\
    &\qquad+C \int_{\Omega}|x|^2\left|\sum_{\beta_1,\beta_2} \sum_{l=1}^{|\beta_1|}f_{\beta_1,l}(x)\frac{\left(\ln r\right)^{(l)}}{\ln L}\cdot D^{\beta_2}w_1\right|^2\md x\\
    &\leq \frac{C}{\ln^4L}\int_{L}^{L^2}t^{-1}\ln^2(t)\md t+\frac{C}{\ln^2L}\int_{\Omega}|x|^2\left|\Delta^j(r^{2j-2}\ln^2 r)\right|^2\md x\\
    &\leq\frac{C}{\ln^4L}\int_{L}^{L^2}t^{-1}\ln^2(t)\md t+\frac{C}{\ln^2L}\int_L^{L^2}t^{-1}\md t\\
    &\to0\qquad\text{as $L\to+\oo$.}
\end{align*}
Therefore, we deduce that
\begin{align*}
    \lim_{L\to+\oo}\int_{\RR^2}|x|^2\left|P_{m-1}\left(u^\star_L-\Phi_{b,m}^\star\right)\right|^2\md x=0.
\end{align*}
Combining \eqref{approximate Phi star 1 when N=2} and \eqref{approximate Phi star 2 when N=2}, we have $\Phi_{b,m}^\star\in \mathcal{H}^m_1$ and 
\begin{align*}
    \lim_{L\to+\oo}\frac{\int_{\RR^N}\left|P_mu_L^\star\right|^2\md x\int_{\RR^N}|x|^2\left|P_{m-1}u_L^\star\right|^2\md x}{\left(\int_{\RR^N}\left|P_{m-1}u_L^\star\right|^2\md x\right)^2}&= \frac{\int_{\RR^N}\left|P_m\Phi_{b,m}^\star\right|^2\md x\int_{\RR^N}|x|^2\left|P_{m-1}\Phi_{b,m}^\star\right|^2\md x}{\left(\int_{\RR^N}\left|P_{m-1}\Phi_{b,m}^\star\right|^2\md x\right)^2}\\
    &=\mu^\star_m(N).
\end{align*}
\end{proof}
\begin{proof}[Proof of Corollary \ref{cor m order HUP without tensor}]\hfil

In fact, for $u\in C^\oo_{0,\text{rad}}\left(\RR^N\right)$, we have $    \mathcal{R}_1 u=u'(r)$ for $r=|x|$ and $\mathcal{R}_2 u=\Delta u$. It follows that $\left|\mathcal{R}_m u\right|=\left|P_{m}u\right|$. Therefore, by Theorem \ref{thm m order HUP without tensor}, we obtain \eqref{equation cor m order HUP without tensor} holds. Moreover, since the extremal functions given in Theorem \ref{thm m order HUP without tensor} are  radially symmetric, they are also extremal functions for \eqref{equation cor m order HUP without tensor}. Hence the constant $\mu_m^\star(N)$ in \eqref{equation cor m order HUP without tensor} is optimal.
\end{proof}
\vspace{2cm}
\section{Proof of Theorem \ref{thm m order HyUP}}\label{sec 7}
In this section, we establish the sharp  Hydrogen uncertainty principle for radially symmetric functions. First, using the “expanding the squares” method, we prove the following higher-order Hydrogen uncertainty principle.
\begin{proposition}
\label{Prop HyUP}
For $N,\,m\in\NN_+$ and $N\geq2$, we have
\begin{align*}
    \int_{\RR^N}\left|\mathcal{R}_m u\right|^2\md x\int_{\RR^N}\left|\mathcal{R}_{m-1}u\right|^2\md x\geq\bar{\mu}_m(N)\left(\int_{\RR^N}\frac{\left|\mathcal{R}_{m-1}u\right|^2}{|x|}\md x\right)^2 \qquad\text{for $u\in\mathcal{H}_3^m$},
\end{align*}
where $\bar{\mu}_m(N)=\dfrac{(N+1)^2}{4}$ if $m$ is even and $\bar{\mu}_m(N)=\dfrac{(N-1)^2}{4}$ if $m$ is odd.
\end{proposition}
\begin{proof}
Denote $j=\lfloor m/2\rfloor$ and $r=|x|$.
\vspace{0.5cm}
\\
\textbf{Case 1:} $m=2j+1$. Let $u\in C^\oo_{0,\text{rad}}\left(\RR^N\right)\backslash\{0\}$ and $v=\mathcal{R}_{m-1}u=\mathcal{R}_2^j u$. For $b>0$, there holds:
\begin{align*}
    0\leq \int_{\RR^N}\left|v'(r)+bv(r)\right|^2\md x=\int_{\RR^N}|v'(r)|^2\md x+b^2\int_{\RR^N}|v(r)|^2\md x+2b\int_{\RR^N}v(r)v'(r)\md x.
\end{align*}
Note that
\begin{align*}
    2b\int_{\RR^N}v(r)v'(r)\md x&=b\left|\mathbb{S}^{N-1}\right|\int_0^{+\oo}2v(r)v'(r)r^{N-1}\md r\\
    &=-b(N-1)\left|\mathbb{S}^{N-1}\right|\int_0^{+\oo}\frac{v^2}{r}r^{N-1}\md r\\
    &=-b(N-1)\int_{\RR^N}\frac{v^2}{r}\md x.
\end{align*}
\\
It follows that
\begin{align*}
    b(N-1)\int_{\RR^N}\frac{v^2}{r}\md x\leq \int_{\RR^N}|v'(r)|^2\md x+b^2\int_{\RR^N}|v(r)|^2\md x
\end{align*}
Since $u\not\equiv0$, we can take
\begin{equation*}
    b=\left(\int_{\RR^N}\left|v'(r)\right|^2\md x\right)^{\frac{1}{2}}\left(\int_{\RR^N}\left|v\right|^2\md x\right)^{-\frac{1}{2}}.
\end{equation*}
Together with the fact that $v=\mathcal{R}_{m-1}u$ and $|v'(r)|=|\mathcal{R}_1\mathcal{R}_{m-1}u|$, we obtain
\begin{align*}
    \left(\frac{N-1}{2}\right)^2\left(\int_{\RR^N}\frac{\left|\mathcal{R}_{m-1}u\right|^2}{|x|}\md x\right)^2\leq\int_{\RR^N}\left|\mathcal{R}_mu\right|^2\md x\int_{\RR^N}\left|\mathcal{R}_{m-1}u\right|^2\md x.
\end{align*}
\vspace{0.5cm}
\\
\textbf{Case 2:} $m=2j$. Let $u\in C^\oo_{0,\text{rad}}\left(\RR^N\right)\backslash\{0\}$ and $v=\mathcal{R}_{m-1}u=\mathcal{R}_1\mathcal{R}_2^{j-1} u$. For $b>0$, there holds:
\begin{align}
    0&\leq \int_{\RR^N}\left|v'(r)-\frac{v(r)}{r}+bv(r)\right|^2\md x\nonumber\\
    &=\int_{\RR^N}\left|v'(r)-\frac{v(r)}{r}\right|^2\md x+b^2\int_{\RR^N}|v(r)|^2\md x+2b\int_{\RR^N}v(r)\left(v'(r)-\frac{v(r)}{r}\right)\md x.
    \label{expand squares HyUP 1}
\end{align}
By integration by parts, we deduce that
\begin{align}
    \int_{\RR^N}\left|v'(r)-\frac{v(r)}{r}\right|^2\md x&=\int_{\RR^N}\left(\left|v'(r)\right|^2+\frac{v^2(r)}{r^2}-2\frac{v(r)v'(r)}{r}\right)\md x\nonumber\\
    &=\int_{\RR^N}\left(\left|v'(r)\right|^2+\frac{v^2(r)}{r^2}+(N-2)\frac{v^2(r)}{r^2}\right)\md x\nonumber\\
    &=\int_{\RR^N}\left(\left|v'(r)\right|^2+(N-1)\frac{v^2(r)}{r^2}\right)\md x\nonumber\\
    &=\int_{\RR^N}\left(\left|v'(r)\right|^2+(N-1)^2\frac{v^2(r)}{r^2}-(N-1)(N-2)\frac{v^2(r)}{r^2}\right)\md x\nonumber\\
    &=\int_{\RR^N}\left(\left|v'(r)\right|^2+(N-1)^2\frac{v^2(r)}{r^2}+(N-1)\frac{2v(r)v'(r)}{r}\right)\md x\nonumber\\
    &=\int_{\RR^N}\left|v'(r)+\frac{N-1}{r}v(r)\right|^2\md x\nonumber\\
    &=\int_{\RR^N}\left|\mathcal{R}_2\circ\mathcal{R}_2^{j-1}u\right|^2\md x. \label{expand squares HyUP 2}\\
    2b\int_{\RR^N}\left|v(r)\left(v'(r)-\frac{v(r)}{r}\right)\right|\md x&=b\int_{\RR^N}2v(r)v'(r)\md x-2b\int_{\RR^N}\frac{v^2(r)}{r}\md x\nonumber\\
    &=-b(N-1)\int_{\RR^N}\frac{v^2(r)}{r}\md x-2b\int_{\RR^N}\frac{v^2(r)}{r}\md x\nonumber\\
    &=-b(N+1)\int_{\RR^N}\frac{v^2(r)}{r}\md x. \label{expand squares HyUP 3}
\end{align}
Substituting \eqref{expand squares HyUP 2} and \eqref{expand squares HyUP 3} into \eqref{expand squares HyUP 1}, we obtain
\begin{align*}
    b(N+1)\int_{\RR^N}\frac{\left|\mathcal{R}_{m-1}u\right|^2}{|x|}\md x\leq \int_{\RR^N}\left|\mathcal{R}_m u\right|^2\md x+b^2\int_{\RR^N}\left|\mathcal{R}_{m-1}u\right|^2\md x.
\end{align*}
Because $u\not\equiv0$, we can choose 
\begin{equation*}
    b=\left(\int_{\RR^N}\left|\mathcal{R}_m u\right|^2\md x\right)^{\frac{1}{2}}\left(\int_{\RR^N}\left|\mathcal{R}_{m-1}u\right|^2\md x\right)^{-\frac{1}{2}}.
\end{equation*}
It follows that
\begin{align*}
    \left(\frac{N+1}{2}\right)^2\left(\int_{\RR^N}\frac{\left|\mathcal{R}_{m-1}u\right|^2}{|x|}\md x\right)^2\leq\int_{\RR^N}\left|\mathcal{R}_mu\right|^2\md x\int_{\RR^N}\left|\mathcal{R}_{m-1}u\right|^2\md x.
\end{align*}

Finally, combining the results in \textbf{Case 1} and \textbf{Case 2},  together with the density argument, we obtain that \eqref{equation m order HyUP} holds in the space $\mathcal{H}^m_3$ for $m\in\NN_+$.
\end{proof}

From the proof above, to find extremal functions, one can start from the two ordinary differential equations $v'(r)+bv(r)=0$ and $v''(r)-r^{-1}v'(r)+bv'(r)=0$ for $b>0$, whose solutions are $g_1(r)=\mathrm{e}^{-br}$ and $g_2(r)=(1+br)\mathrm{e}^{-br}$, respectively. A direct computation shows that
\begin{align*}
    \frac{\int_{\RR^N}\left|\mathcal{R}_1 \left(g_1(|x|)\right)\right|^2\md x\int_{\RR^N}\left|g_1(|x|)\right|^2\md x}{\left(\int_{\RR^N}|x|^{-1}\left|g_1(|x|)\right|^2\md x\right)^2}&=\frac{(N-1)^2}{4},\\
    \frac{\int_{\RR^N}\left|\mathcal{R}_2 \left(g_2(|x|)\right)\right|^2\md x\int_{\RR^N}\left|\mathcal{R}_1\left(g_2(|x|)\right)\right|^2\md x}{\left(\int_{\RR^N}|x|^{-1}\left|\mathcal{R}_1\left(g_2(|x|)\right)\right|^2\md x\right)^2}&=\frac{(N+1)^2}{4}.
\end{align*}
We will then use the operator $\mathcal{I}$ defined by \eqref{def of I} to construct the corresponding extremal functions.
\vspace{1cm}
\begin{proof}[Proof of Theorem \ref{thm m order HyUP}]\hfil

Since the higher-order Hydrogen uncertainty principle in $\mathcal{H}^m_3$ has already been proved in Proposition \ref{Prop HyUP}, we only need to show the existence of the extremal functions. Notice that $g_1(r)=\mathrm{e}^{-br}\in \mathcal{W}(a,\,0)$ and $g_2(r)=(1+br)\mathrm{e}^{-br}\in \mathcal{W}(a,\,0)$ for any $a>0$ and $b>0$. Thus, following the induction argument in Proposition \ref{expression of Ik}, we have
\begin{align*}
    \mathcal{I}^k\left[\mathrm{e}^{-bt}\right](r)&=\sum_{i=1}^{n_k}f_i(r)+Q_{2k-2}(r),\\
    \mathcal{I}^k\left[(1+bt)\mathrm{e}^{-bt}\right](r)&=\sum_{i=1}^{\tilde{n}_k}\tilde{f}_i(r)+\tilde{Q}_{2k-2}(r),
\end{align*}
for large $r>0$, where $k\in\NN_+$, $f_i\in \mathcal{W}\left(N-2k,\,\sigma_i\right)$  with $\sigma_i\in\NN_+\cup\{0\}$ and $\tilde{f}_i\in \mathcal{W}\left(N-2k,\,\tilde{\sigma}_i\right)$ with $\sigma_i,\,\tilde{\sigma}_i\in\NN_+\cup\{0\}$. $Q_{2k-2}$ and $\tilde{Q}_{2k-2}$ are polynomials of degree at most $(2k-2)$  containing only even powers. For $j=\lfloor m/2\rfloor$ and $r=|x|$, we define
\begin{align}
      \Psi_{b,m}(x)=\Psi_{b,m}(r)=\begin{cases}
    \mathrm{e}^{-br},&\text{ if }m=1,\\
    (1+br)\mathrm{e}^{-br},&\text{ if }m=2,\\
    \mathcal{I}^{j}\left[\mathrm{e}^{-bt}\right](r)-Q_{2j-2}(r),&\text{ if } m=2j+1\quad\text{ for }j\geq1,\\
    \mathcal{I}^{j-1}\left[(1+bt)\mathrm{e}^{-bt}\right](r)-\tilde{Q}_{2j-4}(r),&\text{ if } m=2j\quad\text{ for }j\geq2,
    \end{cases}
    \label{def of Psi}
\end{align}
It follows that $\mathcal{R}_{m-1}\Psi_{b,m}=\mathrm{e}^{-b|x|}$ if $m$ is odd and $\mathcal{R}_{m-2}\Psi_{b,m}=(1+b|x|)\mathrm{e}^{-b|x|}$ if $m$ is even. Then we have
\begin{align*}
    \frac{\int_{\RR^N}\left|\mathcal{R}_m\Psi_{b,m}\right|^2\md x\int_{\RR^N}\left|\mathcal{R}_{m-1}\Psi_{b,m}\right|^2\md x}{\left(\int_{\RR^N}|x|^{-1}\left|\mathcal{R}_{m-1}\Psi_{b,m}\right|^2\md x\right)^2}=\bar\mu_m(N)=\begin{cases}
    \frac{(N-1)^2}{4},\qquad\text{if $m$ is odd,}\\
    \frac{(N+1)^2}{4},\qquad\text{if $m$ is even.}
    \end{cases}
\end{align*}

Next we verify that $\Psi_{b,m}$ can be approximated by smooth, compactly supported functions. Note that $\Psi_{b,m}\in C^\oo\left(\RR^N\backslash\{0\}\right)$ is not smooth at the origin. We first derive the asymptotic estimate as $|x|\to+\oo$ and then deal with the smoothness at the origin. By \eqref{def of Psi}, there exists some $\sigma\in\NN_+$ depending on $m$ such that $\Psi_{b,m}(r)\in \mathcal{W}(N-m+1,\,\sigma).$ For any multi-index $\beta$, by \eqref{derivative of radial function}, as $|x|\to+\oo$, there holds
\begin{equation}
    D^\beta\Psi_{b,m}(x)=
    O\left(|x|^{m-1-N-|\beta|}\ln^{\sigma}|x|\right).
    \label{estimate of Psi}
\end{equation}
Let $\bar u_L=\eta_L\Psi_{b,m}$ for $\eta_L$ defined by \eqref{def of eta L} and $L>0$. We will prove the following facts:
\begin{align}
\label{approximate Psi}
    \begin{cases}
    \lim\limits_{L\to+\oo}\int_{\RR^N}\left|\mathcal{R}_{m}\left(\bar u_L-\Psi_{b,m}\right)\right|^2\md x=0,\\
     \lim\limits_{L\to+\oo}\int_{\RR^N}\left|\mathcal{R}_{m-1}\left(\bar u_L-\Psi_{b,m}\right)\right|^2\md x=0,\\
      \lim\limits_{L\to+\oo}\int_{\RR^N}|x|^{-1}\left|\mathcal{R}_{m-1}\left(\bar u_L-\Psi_{b,m}\right)\right|^2\md x=0.
    \end{cases}
\end{align}
Recall that for radially symmetric function $u$, we have $\left|\mathcal{R}_1u\right|=|\nabla u|$ and $\mathcal{R}_2u=\Delta u$. Then we deduce that
\begin{align*}
    \left|\mathcal{R}_{m}\left(\bar u_L-\Psi_{b,m}\right)\right|=\left|{P}_{m}\left(\bar u_L-\Psi_{b,m}\right)\right|=\left|\left(\eta_L-1\right)P_m\Psi_{b,m}+\bar\rho_L\right|,
\end{align*}
where $\bar\rho_L$ is a linear combination of products in the form $D^{\beta_1}\eta_L\cdot D^{\beta_2}\Psi_{b,m}$ for $|\beta_1|+|\beta_2|=m$ and $|\beta_1|\geq1$. From  \eqref{estimate of eta} and \eqref{estimate of Psi}, as $L\to+\oo$, we have
\begin{align*}
    \int_{\RR^N}\left|\mathcal{R}_{m}\left(\bar u_L-\Psi_{b,m}\right)\right|^2\md x\leq C\int_L^{+\oo}t^{-2N-2}\ln^{2\sigma}(t)\cdot t^{N-1}\md t\to0.
\end{align*}
Similarly, we have
\begin{align*}
    \int_{\RR^N}\left|\mathcal{R}_{m-1}\left(\bar u_L-\Psi_{b,m}\right)\right|^2\md x&\leq C\int_L^{+\oo}t^{-2N}\ln^{2\sigma}(t)\cdot t^{N-1}\md t\to0.\\
    \int_{\RR^N}|x|^{-1}\left|\mathcal{R}_{m-1}\left(\bar u_L-\Psi_{b,m}\right)\right|^2\md x&\leq C\int_L^{+\oo}t^{-2N-1}\ln^{2\sigma}(t)\cdot t^{N-1}\md t\to0.
\end{align*}
Then we obtain \eqref{approximate Psi}. Now, we will analyze the behavior of $\Psi_{b,m}$ near the origin. Since $Q_{2j-2}(r)$ and $\tilde{Q}_{2j-4}(r)$ are  polynomials containing only even powers, they are smooth at the origin. Note that $g_1(t)=\mathrm{e}^{-bt}$ and $g_2(t)=(1+bt)\mathrm{e}^{-bt}$ possess power series expansions in some neighborhood of the origin. Then, in some punctured neighborhood $B_\delta(0)\backslash\{0\}$ of the origin, for $j=\lfloor m/2\rfloor$, we have
\begin{align*}
    \begin{cases}
    \mathcal{I}^j\left[\mathrm{e}^{-bt}\right](r)=\sum\limits_{l=2j}\limits^{+\oo}a_l r^{l},\qquad&\text{if $m=2j+1$ for $j\geq1$},\\
    \mathcal{I}^{j-1}\left[(1+bt)\mathrm{e}^{-bt}\right](r)=\bar a_{2j-2}r^{2j-2}+\sum\limits_{l=2j}\limits^{+\oo}\bar a_l r^{l},&\text{if $m=2j$ for $j\geq2$}.
    \end{cases}
\end{align*}
Consequently, for any multi-index $\beta$ with $|\beta|\leq m$, we have that $D^\beta \left(\mathcal{I}^j\left[\mathrm{e}^{-bt}\right]\right)$ is bounded in $B_\delta(0)\backslash\{0\}$ when $m$ is odd, and $D^\beta \left(\mathcal{I}^{j-1}\left[(1+bt)\mathrm{e}^{-bt}\right]\right)$ is bounded in $B_\delta(0)\backslash\{0\}$ when $m$ is even. Hence, for any $R>0$ and $\phi\in C^\oo_0\left(B_R(0)\right)$, there holds:
\begin{align*}
    \int_{B_R(0)}\Psi_{b,m}D^\beta\phi\md x=(-1)^{|\beta|}\int_{B_R(0)}\phi\cdot D^\beta\Psi_{b,m}\md x\qquad\text{for }|\beta|\leq m.
\end{align*}
That is, all weak derivatives of $\Psi_{b,m}$ up to order $m$ exist and are bounded on any ball $B_R(0)$. Then we have $\Psi_{b,m}$ belongs to the Sobolev space $W^{m,\oo}_{\text{loc}}\left(\RR^N\right)$. It follows that $\bar u_L=\eta_L\Psi_{b,m}\in W^{m,\oo}_{\text{loc}}\left(\RR^N\right)$ and we denote its mollification by $\left(\bar u_L\right)_{\varepsilon}$ for small $\varepsilon>0$ (see \cite[Appendix C]{MR2597943}). Then $\left(\bar u_L\right)_{\varepsilon}$ is radially symmetric and belongs to $C^\oo_0\left(B_{4L}(0)\right)$. For fixed $L>0$, by Theorem 1 in \cite[Section 5.3.1]{MR2597943}, we have
\begin{align}
\label{mollification 1}
    \lim_{\varepsilon\to0}\int_{\RR^N}\left|\mathcal{R}_m\left(\left(\bar u_L\right)_{\varepsilon}-\bar u_L\right)\right|^2\md x&=\lim_{\varepsilon\to0}\int_{\RR^N}\left|{P}_m\left(\left(\bar u_L\right)_{\varepsilon}-\bar u_L\right)\right|^2\md x\nonumber\\
    &=\lim_{\varepsilon\to0}\int_{\RR^N}\left|\left({P}_m\bar u_L\right)_{\varepsilon}-{P}_m\bar u_L\right|^2\md x\nonumber\\
    &=0.
    \end{align}
    \begin{align}
    \lim_{\varepsilon\to0}\int_{\RR^N}\left|\mathcal{R}_{m-1}\left(\left(\bar u_L\right)_{\varepsilon}-\bar u_L\right)\right|^2\md x&=\lim_{\varepsilon\to0}\int_{\RR^N}\left|{P}_{m-1}\left(\left(\bar u_L\right)_{\varepsilon}-\bar u_L\right)\right|^2\md x\nonumber\\
    &=\lim_{\varepsilon\to0}\int_{\RR^N}\left|\left({P}_{m-1}\bar u_L\right)_{\varepsilon}-{P}_{m-1}\bar u_L\right|^2\md x\nonumber\\
    &=0.
    \label{mollification 2}
\end{align}
Moreover, by Theorem 7 in \cite[Appendix C]{MR2597943}, we have $\left({P}_{m-1}\bar u_L\right)_{\varepsilon}\to{P}_{m-1}\bar u_L$ as $\varepsilon \to0$ uniformly on compact subsets of $B_{4L}(0)$. Then, we obtain that
\begin{align}
\label{mollification 3}
    \int_{\RR^N}|x|^{-1}\left|\mathcal{R}_{m-1}\left(\left(\bar u_L\right)_{\varepsilon}-\bar u_L\right)\right|^2\md x &\leq \|\left({P}_{m-1}\bar u_L\right)_{\varepsilon}-{P}_{m-1}\bar u_L\|_{L^\oo\left(B_{3L}(0)\right)}^2\int_{B_{3L}(0)}|x|^{-1}\md x\nonumber\\
    &\to 0\qquad\text{as $\varepsilon\to0$}.
\end{align}
Finally, combining \eqref{approximate Psi}, \eqref{mollification 1}, \eqref{mollification 2} and \eqref{mollification 3}, we conclude that there exists a sequence of functions $\left\{\Psi_{n}\right\}\subset C^{\oo}_{0,\text{rad}}\left(\RR^N\right)$ such that
\begin{align*}
   \lim_{n\to+\oo}\frac{\int_{\RR^N}\left|\mathcal{R}_m\Psi_{n}\right|^2\md x\int_{\RR^N}\left|\mathcal{R}_{m-1}\Psi_{n}\right|^2\md x}{\left(\int_{\RR^N}|x|^{-1}\left|\mathcal{R}_{m-1}\Psi_{n}\right|^2\md x\right)^2}&= \frac{\int_{\RR^N}\left|\mathcal{R}_m\Psi_{b,m}\right|^2\md x\int_{\RR^N}\left|\mathcal{R}_{m-1}\Psi_{b,m}\right|^2\md x}{\left(\int_{\RR^N}|x|^{-1}\left|\mathcal{R}_{m-1}\Psi_{b,m}\right|^2\md x\right)^2}\\
   &=\bar\mu_m(N).
\end{align*}
That is, $\Psi_{b,m}\in \mathcal{H}^m_3\,(b>0)$ are extremal functions of \eqref{equation m order HyUP}.
\end{proof}

\vspace{2cm}
 \noindent{\Large\textbf{Acknowledgement}} 
 ~\\
~\\
M. Liu  is funded by NSFC (12601191) and Zhejiang Provincial Natural Science Foundation of China (LQN25A010007) and the Fundamental Research Funds for the Provincial Universities of Zhejiang (GK259909299001-029). W.
Zou is funded by National Key R\&D Program of China (Grant 2023YFA1010001) and NSFC (12171265).

\vspace{2cm}
\noindent{\Large\textbf{Statements and Declarations}}
~\\
~\\
{\textbf{Data Availability Statements}}\quad This article has no additional data.~\\
\\
\textbf{Conflict of interest}\quad There is no conflict of
interest.


\end{document}